\documentclass[10pt]{siamltex}

\usepackage{amsfonts}
\usepackage{graphics}
\usepackage{graphicx}
\usepackage{epsfig}
\usepackage[lofdepth,lotdepth]{subfig}
\usepackage{amsmath}
\usepackage{float}
\usepackage{amssymb}

\usepackage{color}
\usepackage{colortbl}
\usepackage{xcolor}
\usepackage{enumerate}
\definecolor{gray}{rgb}{0.85,0.85,0.85}
\usepackage{multirow}

\newcommand{\new}{\newcommand*}
\newcommand{\newe}{\newenvironment*}
\new{\newt}{\newtheorem}
\new{\ed}{\end{document}}
\newt{assumption}{Assumption}[section]
\newt{define}{Definition}[section]
\newt{example}{Example}[section]
\newt{assertion}{Assertion}[section]
\newe{remark}{{\flushleft \it Remark}~\thesection.}{}
\newe{remarks}{{\flushleft \it Remarks.}}{}
\newe{acknowledgment}{{\flushleft \bf Acknowledgments. }}{}

\def\alf{\alpha}

\def\N{{\scriptscriptstyle N}}
\def\M{{\scriptscriptstyle M}}

 \makeatletter
\def\ps@pprintTitle{%
 \let\@oddhead\@empty
 \let\@evenhead\@empty
 \def\@oddfoot{\centerline{\thepage}}%
 \let\@evenfoot\@oddfoot}
\makeatother
\usepackage[mathlines]{lineno}

\title{Multi-domain Spectral Collocation Method for Variable-Order Nonlinear Fractional Differential Equations\thanks{This work was supported by the MURI/ARO on ``Fractional PDEs for Conservation Laws and Beyond: Theory, Numerics and Applications" (W911NF-15-1-0562).} The research of the first author was partially supported by NSF of China (No. 11661048).
}

\author{
  Tinggang Zhao\footnotemark[2]
  \and
  Zhiping Mao\footnotemark[3] \footnotemark[4]
  \and
  George Em Karniadakis\footnotemark[3]
}

\begin{document}

\markboth{}{}

\footnotetext[2]{School of Mathematics, Lanzhou City University, Lanzhou 730070, China ({tinggang\_zhao@brown.edu}).}
\footnotetext[3]{Division of Applied Mathematics, Brown University, Providence, RI 02912, USA ({zhiping\_mao@brown.edu}, {george\_karniadakis@brown.edu}).}
\footnotetext[4]{Corresponding author.}

\maketitle

\begin{abstract}
Spectral and spectral element methods using Galerkin type formulations are efficient for solving linear fractional PDEs (FPDEs) of constant order but are not efficient in solving nonlinear FPDEs and cannot handle FPDEs with variable-order. In this paper, we present a multi-domain spectral collocation method that addresses these limitations. We consider FPDEs in the Riemann-Liouville sense, and employ Jacobi Lagrangian interpolants to represent the solution in each element. We provide variable-order differentiation formulas, which can be computed efficiently for the multi-domain discretization taking into account the nonlocal interactions. We enforce the interface continuity conditions by matching the solution values at the element boundaries via the Lagrangian interpolants, and in addition we minimize the jump in (integer) fluxes using a penalty method. We analyze numerically the effect of the penalty parameter on the condition number of the global differentiation matrix and on the stability and convergence of the penalty collocation scheme. We demonstrate the effectiveness of the new method for the fractional Helmholtz equation of constant and variable-order using $h-p$ refinement for different values of the penalty parameter. We also solve the fractional Burgers equation with constant and variable-order and compare with solutions obtained with a single domain spectral collocation method.
\end{abstract}



\section{Introduction}

Recently, fractional differential equations (FDEs) have been applied in modeling a wide range of complex physical processes, for example, anomalous transport processes~\cite{Bar01, MetK04, MetK00}, frequency-dependent damping behavior in viscoelastic materials~\cite{DieF99, BagT83, BagT85}, continuum and statistical mechanics~\cite{Mai97}, solid mechanic~\cite{RosS97}, economics~\cite{Bai96}, and so on.
Moreover, FDEs with {\it variable-order} can more accurately describe the anomalous diffusion since the diffusion rate may be time or space dependent, see \cite{Kikuchi1997, Ruiz2004} and references therein.

Close forms of solutions of even linear FDEs, especially those with variable-order, is difficult to obtain. Thus, robust numerical methods are required to discretize the fractional operators.
Proposed numerical methods so far include extensions of the finite difference method (FDM)~\cite{LinX07, LiuAT04, MeeST06, MeeT04, SunW06, TadMS06, TadM07, WanB12, WanW11, ZhaSW11, ZhoTD13, ZhaSH14} as well as of the finite element method (FEM)~\cite{Den08, ForXY11, JiaM11, LiaYTLWL16, WanYZ15, ZheLZ10}, and references therein.
However, since the fractional operators are non-local, then both FDM and FEM, which are local methods, would lose a big advantage that they enjoy for classical PDEs with locally defined derivatives.  In contrast, global methods, such as the spectral method, could  play an important role in developing efficient and highly accurate numerical discretizations for FDEs, see \cite{LiZL12, LiX09, LiX10, TiaDW14, XuH14, MaoCS16, MaoS16, ZenLLBTA14, ZayK14} and references therein.
It is well-known that the solutions of FDEs exhibit end-point singularities even with smooth source terms. This causes
additional difficulties in developing high accuracy methods for FDEs.
To resolve this issue, an efficient and highly accurate spectral method was proposed in \cite{ZayK13} by using poly-fractonomials, which are eigenfunctions of a fractional Sturm-Liouville operator, as basis functions, leading to sparse matrices for simple model equations (linear equation without reaction term) of \textit{constant order}; a rigorous error analysis was established in \cite{CheSW15} showing that the error only depends on the regularity of the right-hand function. This idea was also extended to Riesz FDEs~\cite{MaoCS16} and general two-sided FDEs~\cite{MaoK17}.
However, such approach may  fail for solving more complex FDEs, such as FDEs with a reaction term, \textit{variable-order} FDEs or \textit{nonlinear} problems, due to the fact that the singular behavior of the end points is hard to match with one single set of basis functions.

To this end, a more flexible method, which combines domain decomposition with spectral expansions, namely the spectral element method (SEM, also called multi-domain spectral method), was introduced. For instance, Mao and Shen constructed a high accuracy SEM based on the geometric mesh showing that the error has exponential decay with respect to the square root of the number of degrees of freedom without prior knowledge about the singular behavior~\cite{MaoS17}. Some other SEMs can be found in \cite{ZayK14a, ZayK14b, KhaZK17}.
However, the aforementioned SEMs are based on Galerkin or Petrov-Galerkin type formulations, which are not efficient in dealing with \textit{nonlinear} problems.  Moreover, they cannot handle FDEs of \textit{variable-order}. In contrast, the collocation type method does not suffer from  such issues, and is particularly suitable for the \textit{variable-order} fractional problem.
In \cite{ZenZK15,ZenMK17}, the authors developed a generalized spectral collocation method with tunable accuracy for FDEs of \textit{variable-order} by using the weighted Jacobi polynomials. Although effective, the collocation method relies on empirically tuning the basis functions to capture the singular terms and cannot target special discretization strategies, e.g. exploiting the strength of graded or geometric meshes.
%

In this paper, we aim to develop a stable and highly accurate multi-domain spectral collocation method (MDSCM) for solving \textit{variable-order} nonlinear FDEs in the Riemann-Liouville sense. In particular, we employ Jacobi Lagrangian interpolants to represent the solution in each element. Directly applying the MDSCM to solve FDEs may lead to instabilities. Thus, we employ a penalty technique at the element interfaces by minimizing the jump in (integer) fluxes to stabilize the MDSCM. The penalty technique was also used previously in \cite{XuH14} for FDEs in the Caputo sense.
The main contributions of this work are as follows:
\begin{itemize}
\item We first construct a set of $C^0$ nodal basis functions and derive the corresponding variable-order differentiation matrix for the multi-domain discretization taking into account the nonlocal interactions. We also provide an efficient algorithm to compute the entries of the differentiation matrix by hybridizing the three-term recurrence relation and the Gauss quadrature.
\item We enforce the interface continuity conditions by matching the solution values at the element boundaries via the Lagrangian interpolants, and in addition we minimize the jump in (integer) fluxes using a penalty method. We also analyze numerically the effect of the penalty parameter on the condition number of the global differentiation matrix and the stability of the discretized scheme.
\item We demonstrate the effectiveness of the new method for the fractional Helmholtz equation of constant and variable-order using $h-p$ refinement for different values of the penalty parameter, and solve the fractional Burgers equation with constant and variable order, and compare with solutions obtained using a single-domain spectral collocation method.
\end{itemize}

The paper is organized as follows. In the next section, we present some definitions and properties for variable-order fractional integrals and derivatives. We propose the MDSCM in detail in section \ref{sec:MDSCM}. Moreover, in subsection \ref{sec:sta}, we discuss the eigenvalues of the multi-domain fractional differentiation matrix and introduce the penalty technique to stabilize the MDSCM. We present several numerical examples for the fractional Helmholtz equation and the fractional Burgers equation in section \ref{sec:num}. Finally, we conclude in section \ref{sec:conclusion}.

\section{Preliminaries}
In this section, we review definitions of variable-order fractional integrals and derivatives (\cite{ZenZK15, ZhuLAT09}).
\begin{define}For $\rho(x)>0$, the left and right fractional integrals in the sense of Riemann-Liouville are defined as
$$
{_{a}I_x^{\rho(x)}} v(x)=\frac{1}{\Gamma(\rho(x))}\int_{a}^x\frac{v(y)}{(x-y)^{1-\rho(x)}} dy,\quad x\in[a,b],
$$
and
$$
_{x}I_b^{\rho(x)} v(x)=\frac{1}{\Gamma(\rho(x))}\int_{x}^{b}\frac{v(y)}{(y-x)^{1-\rho(x)}}dy,\quad x\in[a,b],
$$
respectively, where $\Gamma(\cdot)$ is the Euler's Gamma function.
\end{define}
\begin{define}For $k-1<\rho(x)<k$ with $k\in\mathbb{N}$, the left and right Riemann-Liouville derivative of order $\rho(x)$ are defined  as
$$
{_a}D_x^{\rho(x)} v(x)=\frac{1}{\Gamma(k-\rho(x))}\left[\frac{d^k}{d\xi^k}\int_{a}^\xi\frac{v(y)}{(\xi-y)^{\rho(x)-k+1}}dy\right]_{\xi=x},
$$
and
$$
{_x}D_b^\rho(x) v(x)=\frac{(-1)^k}{\Gamma(k-\rho(x))}\left[\frac{d^k}{d\xi^k}\int_{\xi}^{b}\frac{v(y)}{(y-\xi)^{\rho(x)-k+1}}dy\right]_{\xi=x}
$$
for $x\in[a,b]$, respectively.
\end{define}

%

Similar to the constant-order case, there exist some well-known properties for the variable-order case, for instance,  if $n>\rho(x)$ or $n-\rho(x)\notin \mathbb{Z}$,
$$
{_a}D_x^{\rho(x)}(x-a)^n=\frac{\Gamma(n+1)}{\Gamma(n+1-\rho(x))}(x-a)^{n-\rho(x)}.
$$
For $k-1<\alf<k$ with $k\in\mathbb{N}$, let us introduce the following notation:
$$
\widetilde{D}_{a}^{b,\rho(x)}u(x)=\frac{1}{\Gamma(k-\rho(x))}\left[\frac{d^k}{d\xi^k}\int_{a}^b\frac{u(y)}{(\xi-y)^{\rho(x)-k+1}}dy\right]_{\xi=x}, \quad x>b.
$$
It can be verified that
$${_a}D_x^{\rho(x)}u(x)={_b}D_x^{\rho(x)}u(x) + \widetilde{D}_{a}^{b,\rho(x)}u(x), \quad a<b<x.  $$

\section{Multi-domain  fractional differentiation matrix of variable-order}\label{sec:MDSCM}
We introduce in this section the multi-domain fractional differentiation matrix (MDFDM) of variable-order and provide an efficient algorithm to compute its entries. We also present how to minimize the jump in the (integer) fluxes using a penalty method.

\subsection{Multi-domain  fractional differentiation matrix of variable-order}
Let $\Lambda : = (x_L, x_R)$, we first divide the interval $\Lambda$ into $M$ elements, i.e.,
$$
x_L=x_0<x_1<\cdots<x_\M=x_R.
$$
Denote $I_k=[x_{k-1},x_{k}],k=1,\ldots,M$ the $k$-th element and $h_k=x_k-x_{k-1}$ the length of $I_k$.
Let $\mathbb{P}_\N^I$ be the collection of all algebraic polynomials defined on interval $I$ with degree at most $N$.
We now introduce the piecewise polynomial space
$$
\mathbb{V}_\N=\{v\in C(\Lambda): v|_{I_k}\in \mathbb{P}_{\N_k}^{I_k}\},
$$
where $N_k, k=1,\ldots,M$ are all positive integers.
For each $k=1,2,\ldots,M$, we select a set of collocation points in $I_k$, denoted by $\{x_j^{k}\}_{j=0}^{\N_k}$, satisfying $x_0^{k}=x_{k-1}$ and $x_{\N_k}^{k}=x_{k}$.
We collect all these points and denote
$$\mathbb{N}_o:=\{x_{j}^{k}: k=1,\ldots,M;j=0,\ldots,N_k\}.$$
The magnitude of $\mathbb{N}_o$ is $\sum_{k=1}^\M N_k+1$.

For $k=1,2,\ldots,M$ and $j=0,\ldots,N_k$, we denote by $L_{j,k}(x)$ the $j$-th Lagrange interpolation polynomial on element $I_k$ satisfying $L_{j,k}(x_i^k) = \delta_{ij}$. Let us first define a set of basis functions.
For the boundary points and interface points, the corresponding basis functions are given by
$$
\phi_0(x)=\left\{
             \begin{array}{ll}
              L_{0,1}(x), & \mbox{~if~}  x\in I_1, \\
               0, & \mbox{~otherwise,~}
             \end{array}
           \right.
\phi_\M(x)=\left\{
             \begin{array}{ll}
              L_{\N_\M,\M}(x), & \mbox{~if~}  x\in I_\M, \\
               0, & \mbox{~otherwise,~}
             \end{array}
           \right.
$$
and
$$
\phi_k(x)=\left\{
             \begin{array}{ll}
              L_{\N_{k},{k}}(x), & \mbox{~if~}  x\in I_k,\\
               L_{0,{k+1}}(x), & \mbox{~if~}  x\in I_{k+1},  \\
               0, & \mbox{~otherwise,}
             \end{array}
           \right.
$$
$k=1,\ldots,M-1$, respectively,
while  for the interior points of each element, the basis functions are given by
$$
\psi_j^k(x)=\left\{
             \begin{array}{ll}
              L_{j,k}(x), & \mbox{~if~}  x\in I_k, \\
               0, & \mbox{~otherwise,}
             \end{array}
           \right.
           k=1,\ldots,M,\; j=1,\ldots,N_k-1.
$$
%
%
Therefore, we have
$$\mathbb{V}_\N=\mbox{span}\{\phi_k,k=0,\ldots,M\} \cup \mbox{span} \{\psi_j^k,k=1,\ldots,M;j=1,\ldots,N_k-1\}.$$
For $u_\N\in \mathbb{V}_\N$, it can be expanded as
\begin{equation}\label{exp-2}
u_\N(x)=\sum_{k=1}^\M \sum_{j=1}^{\N_k-1}u_\N(x_j^k)\psi_j^k(x)+\sum_{k=0}^\M u_\N(x_k)\phi_k(x).
\end{equation}
Taking the fractional derivative of order $\alf(x)$ and evaluating the values at all collocation points, we obtain the MDFDM of $\alf(x)$, denoted by $\mathbf{D}^{\alf}$, as follows
\begin{equation}\label{eqn:MDFDM}
\mathbf{D}^{\alf}=
\begin{bmatrix}
 \widehat{\mathbf{D}}^{11}  &  &   &  & \overline{\mathbf{D}}^{1}\\
\widehat{\mathbf{D}}^{21}  &  \widehat{\mathbf{D}}^{22} &   &  &  \overline{\mathbf{D}}^{2}\\
\vdots&\vdots&\ddots& & \vdots\\
\widehat{\mathbf{D}}^{\M1}  & \widehat{\mathbf{D}}^{\M2} & \cdots &   \widehat{\mathbf{D}}^{\M\M} & \overline{\mathbf{D}}^{\M}\\
\underline{\mathbf{D}}^{1}  & \underline{\mathbf{D}}^{2} & \cdots  &  \underline{\mathbf{D}}^{\M} & \widetilde{\mathbf{D}}
\end{bmatrix},
\end{equation}
where $\widehat{\mathbf{D}}^{ij}$ denotes the differentiation matrix associated with the inner points of the $i$-th and $j$-th elements given by
$$\widehat{\mathbf{D}}^{ij}=\left[{_{x_L}}D^{\alf(x_m^j)}_x\psi_n^i(x_m^j)\right]_{m,n=1,\ldots,\N_i-1},\; i,j=1,\ldots,M.$$
$\overline{\mathbf{D}}^{i}$ and $\underline{\mathbf{D}}^{i}$ denote the differentiation matrices associated with the inner points of the $i$-th element  and interface points given by $$\overline{\mathbf{D}}^{i}=\left[{_{x_L}}D^{\alf(x_m^i)}_x \phi_n(x_m^i)\right]_{m=1,\ldots,\N_i-1;n=1,\ldots,\M-1}$$
and
$$ \underline{\mathbf{D}}^{i}=\left[{_{x_L}}D^{\alf(x_m)}_x \psi_n^i(x_m)\right]_{m=1,\ldots,\M-1;n=1,\ldots,\N_i-1},i=1,\ldots,M,$$
respectively,
and $\widetilde{\mathbf{D}}$ denotes the differentiation matrix associated with the interface points given by
$$\widetilde{\mathbf{D}}=\left[{_{x_L}}D^{\alf(x_m)}_x \phi_n(x_m)\right]_{m,n=1,\ldots,\M-1}.$$

We show the structure of the differentiation matrix $\mathbf{D}^{\alf}$  with $M=5,N_i=4,i=1,\ldots,5$ in Fig.~\ref{fig2};  the entries of the white blocks are zeros while the entries of colored blocks are nonzero. The entries with different colors are evaluated by different formulas proposed in the next subsection.
\begin{figure}[h]
\includegraphics[width=12cm]{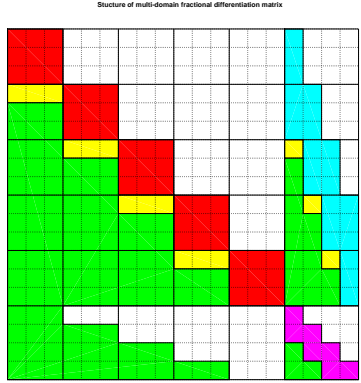}
\caption{Structure of multi-domain fractional differentiation matrix with $M=5,N_i=4,i=1,\ldots,5$.}
\label{fig2}
\end{figure}
All the nodal basis functions defined above are  $C[x_L,x_R]$. The following result shows that $\lim\limits_{x\rightarrow x_i+} {_{x_L}}D_x^{\alf}\phi_i(x)=\infty,\,i=1,\ldots,M-1$ even for a constant $\alpha > 1$. In our case, we only use values of the fractional derivatives of basis functions from left at the interface points, i.e., $\lim\limits_{x\rightarrow x_i-} {_{x_L}}D_x^{\alf}\phi_i(x),\,i=1,\ldots,M-1$, when $\alf>1.$
\begin{lemma}Let $\alf\in (0,2)$ be a constant, $a<c<b$. If $u\in C^2[a,c]\cap C^2[c,b]\cap C[a,b]$ and $u'(c^-)$ and $u'(c^+)$ exist, then we have
\begin{align*}
{_a}D_x^{\alf}u(x)=&\frac{(x-a)^{-\alf}}{\Gamma(1-\alf)}u(a)+\frac{(x-a)^{1-\alf}}{\Gamma(2-\alf)}u'(a)\\
&+\frac{(x-c)^{1-\alf}}{\Gamma(2-\alf)}\left[u'(c^+)-u'(c^-)\right]+s(x),
\end{align*}
for any $x\in (c,b)$, where
$$
s(x)=\left\{
             \begin{array}{ll}
             \frac{1}{\Gamma(3-\alf)}\frac{d}{dx}\left[\int_a^x u''(\tau)(x-\tau)^{2-\alf}d\tau\right], & \mbox{~if~}  \alf\in(0,1), \\
             \frac{1}{\Gamma(4-\alf)}\frac{d^2}{dx^2}\left[\int_a^x u''(\tau)(x-\tau)^{3-\alf}d\tau\right], & \mbox{~if~}  \alf\in(1,2).
             \end{array}
           \right.
$$
\end{lemma}
\begin{proof}
For $\alf\in(0,1)$, using the integration by parts, we have
\begin{align*}
{_a}D_x^{\alf}u(x)=& \frac{1}{\Gamma(1-\alf)}\frac{d}{dx}\left[\int_a^c(x-\tau)^{-\alf}u(\tau)d\tau+\int_c^x(x-\tau)^{-\alf}u(\tau)d\tau\right]\\
=&\frac{1}{\Gamma(1-\alf)}\frac{d}{dx}\Big[\frac{(x-a)^{1-\alf}u(a)-(x-c)^{1-\alf}u(c^-)}{1-\alf}\\
&+\frac{(x-a)^{2-\alf}u'(a)-(x-c)^{2-\alf}u'(c^-)}{(1-\alf)(2-\alf)}+
\frac{\int_a^c(x-\tau)^{2-\alf}u''(\tau)d\tau}{(1-\alf)(2-\alf)}\\
&+\frac{(x-c)^{1-\alf}u(c^+)}{1-\alf}
+\frac{(x-c)^{2-\alf}u'(c^+)}{(1-\alf)(2-\alf)}+
\frac{\int_a^c(x-\tau)^{2-\alf}u''(\tau)d\tau}{(1-\alf)(2-\alf)}\Big].\\	
\end{align*}
Then, the result follows from $u(c^-)=u(c^+)$ for $\alpha\in (0,1)$. For $\alf\in(1,2)$, we obtain the result by using a similar argument.
\end{proof}


\subsection{Computing the differentiation matrix}
We now present how to compute the entries of the MDFDM, i.e., how to perform the computation of ${_{x_L}}D^{\alf(x)}_x \phi_k(x)$, $k=0,1,\ldots,M$ and ${_{x_L}}D^{\alf(x)}_x \psi_j^k(x)$ for $k=1,2,\dots,M,\, j=1,2,\ldots,N_k-1,\; x\in \mathbb{N}_o.$

By the definition of nodal basis functions, we have for $k=1,\ldots,M-1$,
$${_{x_L}}D^{\alf(x)}_x \phi_k(x)=\left\{
             \begin{array}{ll}
             0, & \mbox{~if~} x\leq x_{k-1},\\
              {_{x_{k-1}}}D^{\alf(x)}_x L_{N_k,k}(x), & \mbox{~if~}  x_{k-1}<x\leq x_k, \\
              \widetilde{D}_{x_{k-1}}^{x_k,\alpha(x)}L_{N_k,k}(x)+{_{x_{k}}}D^{\alf(x)}_x L_{0,k+1}(x), & \mbox{~if~} x_k<x\leq x_{k+1},\\
              \widetilde{D}_{x_{k-1}}^{x_k,\alpha(x)}L_{N_k,k}(x)+\widetilde{D}_{x_{k}}^{x_{k+1,\alpha(x)}}L_{0,k+1}(x), & \mbox{~if~} x>x_{k+1},\\
             \end{array}
           \right.
$$
and
\begin{align*}
  {_{x_L}}D^{\alf(x)}_x \phi_0(x) &= \left\{
             \begin{array}{ll}
              {_{x_{0}}}D^{\alf(x)}_x L_{0,1}(x), & \mbox{~if~}  x_0<x\leq x_1, \\
              \widetilde{D}_{x_{0}}^{x_1,\alpha(x)}L_{0,1}(x), & \mbox{~if~} x>x_1,
             \end{array}
           \right. \\
  {_{x_L}}D^{\alf(x)}_x \phi_M(x) &= \left\{
             \begin{array}{ll}
              0, & \mbox{~if~}  x\leq  x_{M-1}, \\
              {_{x_{M-1}}}D^{\alf(x)}_xL_{N_M,M}(x), & \mbox{~if~} x_{M-1} <x\leq x_M.
             \end{array}
           \right.
\end{align*}
Moreover, for $k=1,\ldots,M$ and $j=1,\ldots,N_k$, we have
$${_{x_L}}D^{\alf(x)}_x \psi_j^k(x)=\left\{
             \begin{array}{ll}
             0, & \mbox{~if~} x\leq x_{k-1},\\
              {_{x_{k-1}}}D^{\alf(x)}_x L_{j,k}(x), & \mbox{~if~}  x_{k-1}<x\leq x_k, \\
              \widetilde{D}_{x_{k-1}}^{x_k,\alpha(x)}L_{j,k}(x), & \mbox{~if~} x>x_k.
             \end{array}
           \right.
$$
Overall, we observe that we shall compute the following two types of integral:
$$
{_{x_{k-1}} D^{\alpha(x)}_x} L_{j,k}(x),~ x\in(x_{k-1},x_k]\quad\mbox{and}\quad\widetilde{D}_{x_{k-1}}^{x_k,\alpha(x)}L_{j,k}(x),\quad x>x_k.
$$
By using the transformation $x=\frac{h_k}{2}(y+1)+x_{k-1}\in I_k$,
we arrive at
\begin{equation}\label{tran_fd}
{_{x_{k-1}} D^{\alpha(x)}_x} L_{j,k}(x)=\left(\frac{2}{h_k}\right)^{\hat{\alpha}(y)} {_{-1} D^{\hat{\alpha}(y)}_y} {L}_{j,k}(y)
\end{equation}
and
\begin{equation}\label{tran_fd_ov}
{\widetilde{D}^{x_{k},{\alpha(x)}}_{x_{k-1}}} L_{j,k}(x)=\left(\frac{2}{h_k}\right)^{\hat{\alpha}(y)} {\widetilde{D}^{1,{\hat{\alpha}(y)}}_{-1}} {L}_{j,k}(y),
\end{equation}

For the sake of simplicity, we drop the index $k$ and let $L_j(y)$ be the $j$-th Lagrange interpolation polynomial associated with nodes $\{y_i\}_{i=0}^\N\in [-1,1]$.
%
The following results play an important role in \textit{efficiently} computing the above two types of integral. The first one
can be obtained by replacing $\alf$ with $k-\alf$ in the equation (3.9) of \cite{LiZL12}.
%
\begin{theorem}
For $k-1<\alf(y)<k,\; k\in\mathbb{N}$, let
\begin{equation*}
\hat{R}_j^{c,d,\alf(y)}(y):=\frac{1}{\Gamma(k-\alf(y))}\int_{-1}^y\frac{P_j^{c,d}(s)}{(y-s)^{\alf(y)-k+1}}ds,\quad y\in [-1,1].
\end{equation*}
Then, $\{\hat{R}_j^{c,d,\alf(y)}(y)\}_{j=0}^\N$ satisfies the following three-term recurrence relation:
\begin{align*}
\hat{R}_0^{c,d,\alf(y)}(y)=&\frac{(y+1)^{k-\alf(y)}}{\Gamma(k-\alf(y)+1)},\\
\hat{R}_1^{c,d,\alf(y)}(y)=&
P^{c,d}_1(-1)\frac{(y+1)^{k-\alf(y)}}{\Gamma(k-\alf(y)+1)}+
\frac{c+d+2}{2}\frac{(y+1)^{k-\alf(y)+1}}{\Gamma(k-\alf(y)+2)},\\
 \hat{R}_{j+1}^{c,d,\alf(y)}(y)=&\left(\widetilde{A}_j^{c,d,\alf(y)}y-\widetilde{B}_j^{c,d,\alf(y)}\right) \hat{R}_{j}^{c,d,\alf(y)}(y)\\
 & -\widetilde{C}_j^{c,d,\alf(y)}\hat{R}_{j-1}^{c,d,\alf(y)}(y)+\frac{\widetilde{D}_j^{c,d,\alf(y)}}{\Gamma(k-\alf(y))}(y+1)^{k-\alf(y)},
\end{align*}
for $j\geq1$, where
\begin{equation}\label{eqn:ABCD}
\begin{aligned}
 &\widetilde{A}_j^{c,d,\alf(y)}=\frac{A_j^{c,d}}{1+(k-\alf(y)) A_j^{c,d}\widehat{C}_j^{c,d}},\\
 &\widetilde{B}_j^{c,d,\alf(y)}=\frac{B_j^{c,d}+(k-\alf(y)) A_j^{c,d}\widehat{B}_j^{c,d}}{1+(k-\alf(y)) A_j^{c,d}\widehat{C}_j^{c,d}},\\
 &\widetilde{C}_j^{c,d,\alf(y)}=\frac{C_j^{c,d}+(k-\alf(y)) A^{c,d}_j\widehat{A}_j^{c,d}}{1+(k-\alf(y)) A_j^{c,d}\widehat{C}_j^{c,d}},\\
 &\widetilde{D}_j^{c,d,\alf(y)}=\frac{A_j^{c,d}\left(\widehat{A}_j^{c,d}P_{j-1}^{c,d}(-1)+\widehat{B}_j^{c,d}P_j^{c,d}(-1)
+\widehat{C}_j^{c,d}P_{j+1}^{c,d}(-1)\right)}{1+(k-\alf(y)) A^{c,d}_j\widehat{C}_j^{c,d}},
\end{aligned}
\end{equation}
$A_j^{c,d},B_j^{c,d}$ and $C_j^{c,d}$ can be found in \cite[Equation (3.111)]{SheTW11}, $\widehat{A}_j^{c,d},\widehat{B}_j^{c,d}$
and $\widehat{C}_j^{c,d}$ can be found in \cite[Equation (3.124)]{SheTW11}.
\end{theorem}

The next theorem is an extension of \cite[Theorem 3.1]{CheXH15} for the {\it variable-order} case.

\begin{theorem}
For $k-1<\alf(y)<k$ and $k\in\mathbb{N}$, let
\begin{equation*}
\breve{R}_j^{c,d,\alf(y)}(y):=\frac{1}{\Gamma(k-\alf(y))}\int_{-1}^1\frac{P_j^{c,d}(s)}{(y-s)^{\alf(y)-k+1}}ds,\quad y>1.
\end{equation*}
Then, $\{\breve{R}_j^{c,d,\alf(y)}(y)\}_{j=0}^\N$ satisfies the following three-term recurrence relation:
\begin{align*}
\breve{R}_0^{c,d,\alf(y)}(y)=&\frac{(y+1)^{k-\alf(y)}-(y-1)^{k-\alf(y)}}{\Gamma(k-\alf(y)+1)},\\
\breve{R}_1^{c,d,\alf(y)}(y)=&
\frac{P^{c,d}_1(-1)(y+1)^{k-\alf(y)}-P^{c,d}_1(1)(y-1)^{k-\alf(y)}}{\Gamma(k-\alf(y)+1)}\\
& +\frac{(c+d+2)\left((y+1)^{k-\alf(y)+1}-(y-1)^{k-\alf(y)+1}\right)}{2\Gamma(k-\alf(y)+2)},\\
 \breve{R}_{j+1}^{c,d,\alf(y)}(y)=&\left(\widetilde{A}_j^{c,d,\alf(y)}y-\widetilde{B}_j^{c,d,\alf(y)}\right) \breve{R}_{j}^{c,d,\alf(y)}(y)-\widetilde{C}_j^{c,d,\alf(y)}\breve{R}_{j-1}^{c,d,\alf(y)}(y)\\
 & +\frac{\widetilde{D}_j^{c,d,\alf(y)}}{\Gamma(k-\alf(y))}(y+1)^{k-\alf(y)}-\frac{\widetilde{E}_j^{c,d,\alf(y)}}{\Gamma(k-\alf(y))}(y-1)^{k-\alf(y)},
\end{align*}
for $j\geq1$, where $\widetilde{A}_j^{c,d,\alf(y)}, \widetilde{B}_j^{c,d,\alf(y)},\widetilde{C}_j^{c,d,\alf(y)},\widetilde{D}_j^{c,d,\alf(y)} $ are given in \eqref{eqn:ABCD} and
\begin{equation}\label{eqn:E}
\widetilde{E}_j^{c,d,\alf(y)}=\frac{A_j^{c,d}\left(\widehat{A}_j^{c,d}P_{j-1}^{c,d}(1)+\widehat{B}_j^{c,d}P_j^{c,d}(1)
+\widehat{C}_j^{c,d}P_{j+1}^{c,d}(1)\right)}{1+(k-\alf(y)) A^{c,d}_j\widehat{C}_j^{c,d}}.
\end{equation}
\end{theorem}

We next compute the integer derivatives of $\hat{R}_j^{c,d,\alf(y)}(y)$ and $\breve{R}_j^{c,d,\alf(y)}(y)$ up to order $k$. To this end, we have the following results:
\begin{theorem}\label{them:3r:type1}
Let $k-1<\alf(y)<k,~1\leq m\leq k$ and $k,m\in\mathbb{N}$, we have
the following three-term recurrence relation: for $j\geq1$,
\begin{align*}
\frac{d^m}{dy^m}\hat{R}_0^{c,d,\alf(y)}(y) = &\frac{(y+1)^{k-\alf(y)-m}}{\Gamma(k-\alf(y)+1-m)},\\
\frac{d^m}{dy^m}\hat{R}_1^{c,d,\alf(y)}(y) =&
P^{c,d}_1(-1)\frac{(y+1)^{k-\alf(y)-m}}{\Gamma(k-\alf(y)+1-m)}+
\frac{c+d+2}{2}\frac{(y+1)^{k-\alf(y)+1-m}}{\Gamma(k-\alf(y)+2-m)},\\
\frac{d^m}{dy^m} \hat{R}_{j+1}^{c,d,\alf(y)}(y) =& \left(\widetilde{A}_j^{c,d,\alf(y)}y-\widetilde{B}_j^{c,d,\alf(y)}\right) \frac{d^m}{dy^m}\hat{R}_{j}^{c,d,\alf(y)}(y)\\
& -\widetilde{C}_j^{c,d,\alf(y)}\frac{d^m}{dy^m}\hat{R}_{j-1}^{c,d,\alf(y)}(y) +m\widetilde{A}_j^{c,d,\alf(y)}\frac{d^{m-1}}{dy^{m-1}}\hat{R}_{j}^{c,d,\alf(y)}(y)\\
& +\frac{(k-\alf(y))\widetilde{D}_j^{c,d,\alf(y)}}{\Gamma(k-\alf(y)+1-m)}(y+1)^{k-\alf(y)-m},
\end{align*}
where $\widetilde{A}_j^{c,d,\alf(y)}, \widetilde{B}_j^{c,d,\alf(y)},\widetilde{C}_j^{c,d,\alf(y)},\widetilde{D}_j^{c,d,\alf(y)} $ are given in \eqref{eqn:ABCD}.
\end{theorem}

\begin{theorem}\label{them:3r:type2}
Let $k-1<\alf(y)<k,~1\leq m\leq k$ and $k,m\in\mathbb{N}$, we have
the following three-term recurrence relation: for $j\geq1$,
\begin{align*}
\frac{d^m}{dy^m}\breve{R}_0^{c,d,\alf(y)}(y) =& \frac{(y+1)^{k-\alf(y)-m}-(y-1)^{k-\alf(y)-m}}{\Gamma(k-\alf(y)+1-m)},\\
\frac{d^m}{dy^m}\breve{R}_1^{c,d,\alf(y)}(y) =&
\frac{P^{c,d}_1(-1)(y+1)^{k-\alf(y)-m}-P^{c,d}_1(1)(y-1)^{k-\alf(y)-m}}{\Gamma(k-\alf(y)+1-m)}\\
& + \frac{(c+d+2)\left((y+1)^{k-\alf(y)+1-m}-(y-1)^{k-\alf(y)+1-m}\right)}{2\Gamma(k-\alf(y)+2-m)},\\
\frac{d^m}{dy^m}\breve{R}_{j+1}^{c,d,\alf(y)}(y) = &\left(\widetilde{A}_j^{c,d,\alf(y)}y-\widetilde{B}_j^{c,d,\alf(y)}\right)
\frac{d^m}{dy^m}\breve{R}_{j}^{c,d,\alf(y)}(y)\\
& - \widetilde{C}_j^{c,d,\alf(y)}\frac{d^m}{dy^m}\breve{R}_{j-1}^{c,d,\alf(y)}(y) + m\widetilde{A}_j^{c,d,\alf(y)}\frac{d^{m-1}}{dy^{m-1}}\breve{R}_{j}^{c,d,\alf(y)}(y)\\
& + \frac{\widetilde{D}_j^{c,d,\alf(y)}(y+1)^{k-\alf(y)-m}-\widetilde{E}_j^{c,d,\alf(y)}(y-1)^{k-\alf(y)-m}}
{(k-\alf(y))^{-1}\Gamma(k-\alf(y)+1-m)},
\end{align*}
where $\widetilde{A}_j^{c,d,\alf(y)}, \widetilde{B}_j^{c,d,\alf(y)},\widetilde{C}_j^{c,d,\alf(y)},\widetilde{D}_j^{c,d,\alf(y)} $ are given in \eqref{eqn:ABCD} and $\widetilde{E}_j^{c,d,\alf(y)}$ is given in \eqref{eqn:E}.
\end{theorem}

We now show how to compute  ${_{-1} D^{\alpha}_y} L_{j}(y)$ and $\widetilde{D}^{1,{\alpha}}_{-1} L_{j}(y)$. To do this, we first expand $L_j(y)$ as
\begin{equation}\label{exL-J}
L_j(y)=\sum_{i=0}^\N l_i^j P_i^{c,d}(y),
\end{equation}
where
$$l_i^j = \left\{
\begin{array}{ll}
P_i^{c,d}(y_j)w_j/\gamma_i^{c,d}, \quad & i = 0,\ldots,N-1, \\
P_N^{c,d}(y_j)w_j/\left((2+\frac{c+d+1}{N})\gamma_i^{c,d}\right), \quad & i = N,
\end{array}
\right.
$$
where $\{y_j,w_j\},j = 0,\ldots,N$ are the nodes and weights of the Jacobi-Gauss-Lobatto quadrature and $\gamma_i^{c,d}$ can be found in \cite[Equation (3.109)]{SheTW11}. Then, we have
$$ {_{-1}D_y^\alf} L_j(y)=\sum_{i=0}^\N l_i^j {_{-1}D_y^\alf}P^{c,d}_i(y).$$
Hence, we only need to compute ${_{-1}D_y^\alf}P^{c,d}_i(y)$, which can be computed by using the three-term recurrence relation proposed in Theorem \ref{them:3r:type1}. The ``red-block" entries in Fig.~\ref{fig2} are evaluated by using this formula.

As for $\widetilde{D}_{-1}^{1,\alf}L_j(y)$, we apply the hybrid approach similar as in \cite{CheXH15}. In particular,
when $y$ is close to 1, we use the following formula
$$ \widetilde{D}^{1,{\alpha}}_{-1} L_j(y)=\sum_{i=0}^\N l_i^j \widetilde{D}^{1,{\alpha}}_{-1} P^{c,d}_i(y),$$
and $\widetilde{D}^{1,{\alpha}}_{-1} P^{c,d}_i(y),i = 0,\ldots,N$ are computed by using the three-term recurrence relation proposed in Theorem \ref{them:3r:type2}. The ``yellow-block" entries in Fig.~\ref{fig2} are evaluated by using this formula.
When $y$ is far away from 1, we use the following Jacobi-Gauss-type quadrature
$$\widetilde{D}_{-1}^{1,\alf}L_j(y)=\frac{1}{\Gamma(-\alf)}
\sum_{i=0}^\N\sum_{k=0}^L l_i^j\frac{P_i^{c,d}(\xi_{k,L})}{(y-\xi_{k,L})^{\alf+1}}\omega_{k,L}^{c,d},$$
where $\xi_{k,L}$ and $\omega_{k,L}$ are the Legendre-Gauss-type quadrature points and weights. The ``green-block" entries in Fig.~\ref{fig2} are evaluated by using this formula.

\subsection{Minimize the jump in the fluxes using a penalty method}\label{sec:sta}
Discontinuities of the (fractional or integer) fluxes of the nodal basis functions at the interfaces may lead to an unstable scheme when solving FDEs. Thus, we introduce a penalty parameter to stabilize the corresponding MDSCM by minimizing the jump in the integer fluxes.

To illustrate the possibility of the instability of the corresponding MDSCM, we show the distribution of eigenvalues of the MDFDM for different values of constant-order $\alf = 1.01,1.99$ with different $M,N$ using a {\it uniform mesh}. The other parameters are taken as $c=d=0, x_L=-1, x_R=1$.
We observe from Figs.~\ref{fig3}-\ref{fig4} that there exist eigenvalues whose real parts are positive, which would cause instability in the algorithm,  for time-dependent problems.

\begin{figure}[!t]
\centering
\includegraphics[width=0.46\textwidth,height=0.4\textwidth]{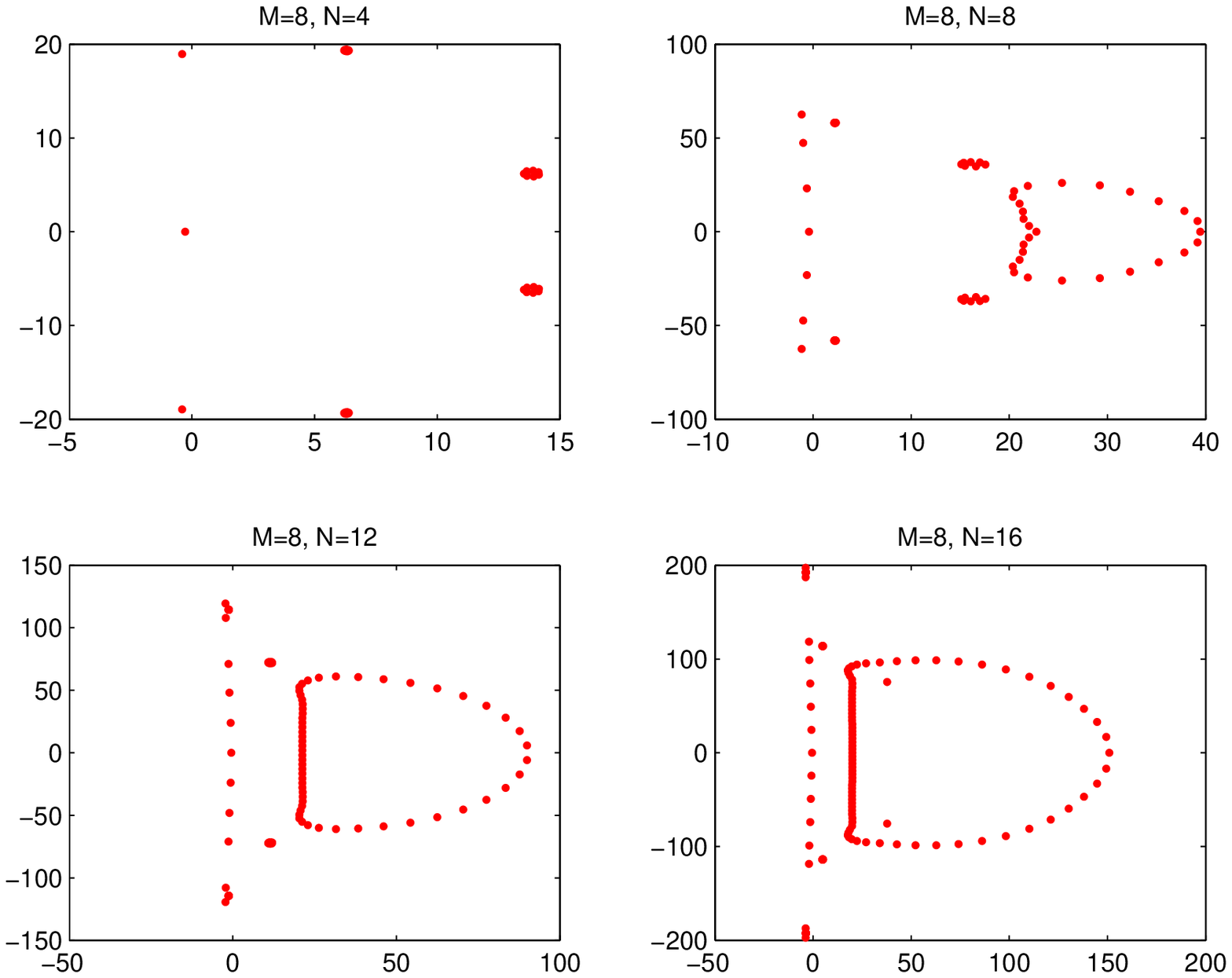}
\quad \;
\includegraphics[width=0.46\textwidth,height=0.4\textwidth]{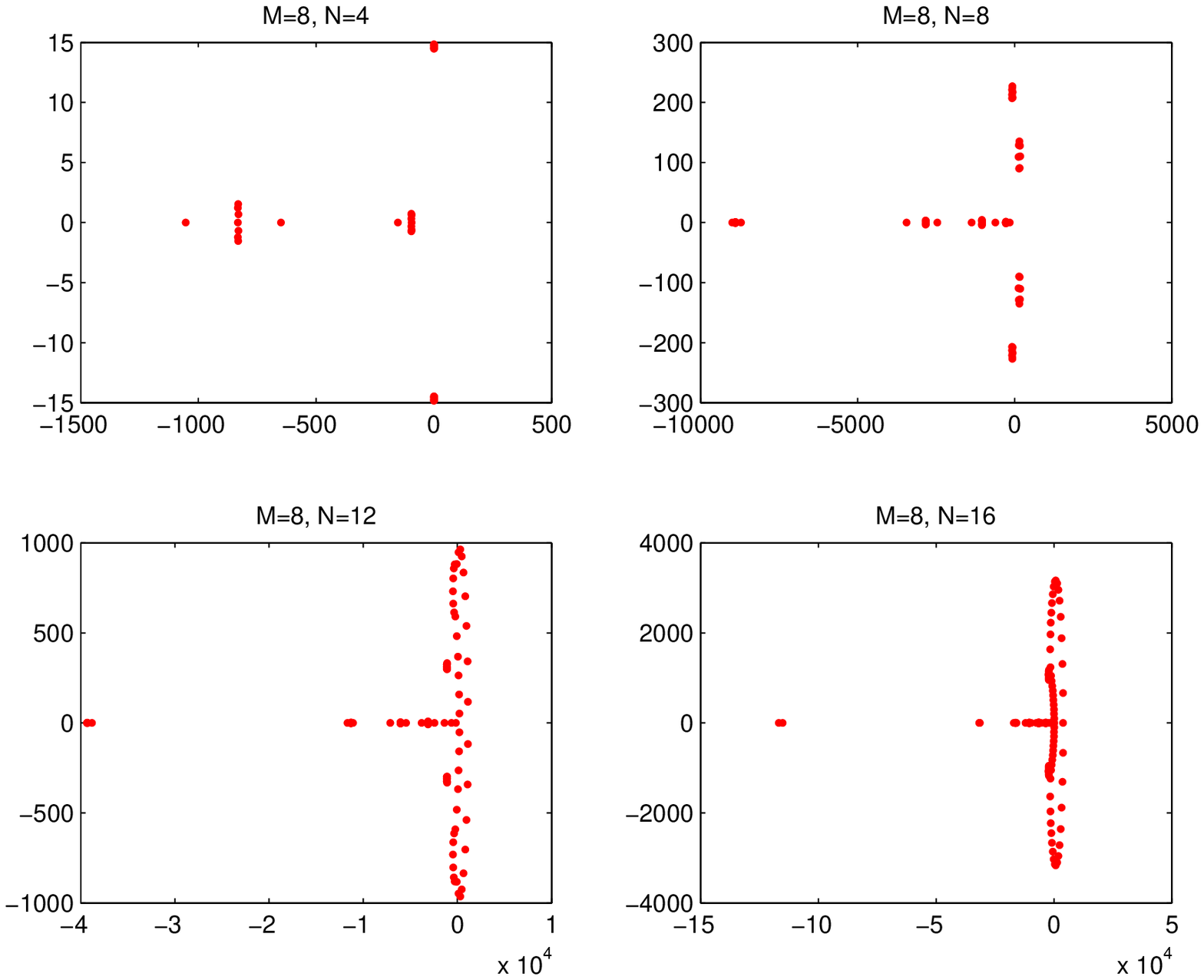}
\caption{Eigenvalues of MDFDM $\mathbf{D}^{\alf}$ for $M = 8$ and different values of $N$ with the uniform mesh. Left: $\alf=1.01$, right: $\alf=1.99$.}
\label{fig3}
\end{figure}

\begin{figure}[!t]
\centering
\includegraphics[width=0.46\textwidth,height=0.4\textwidth]{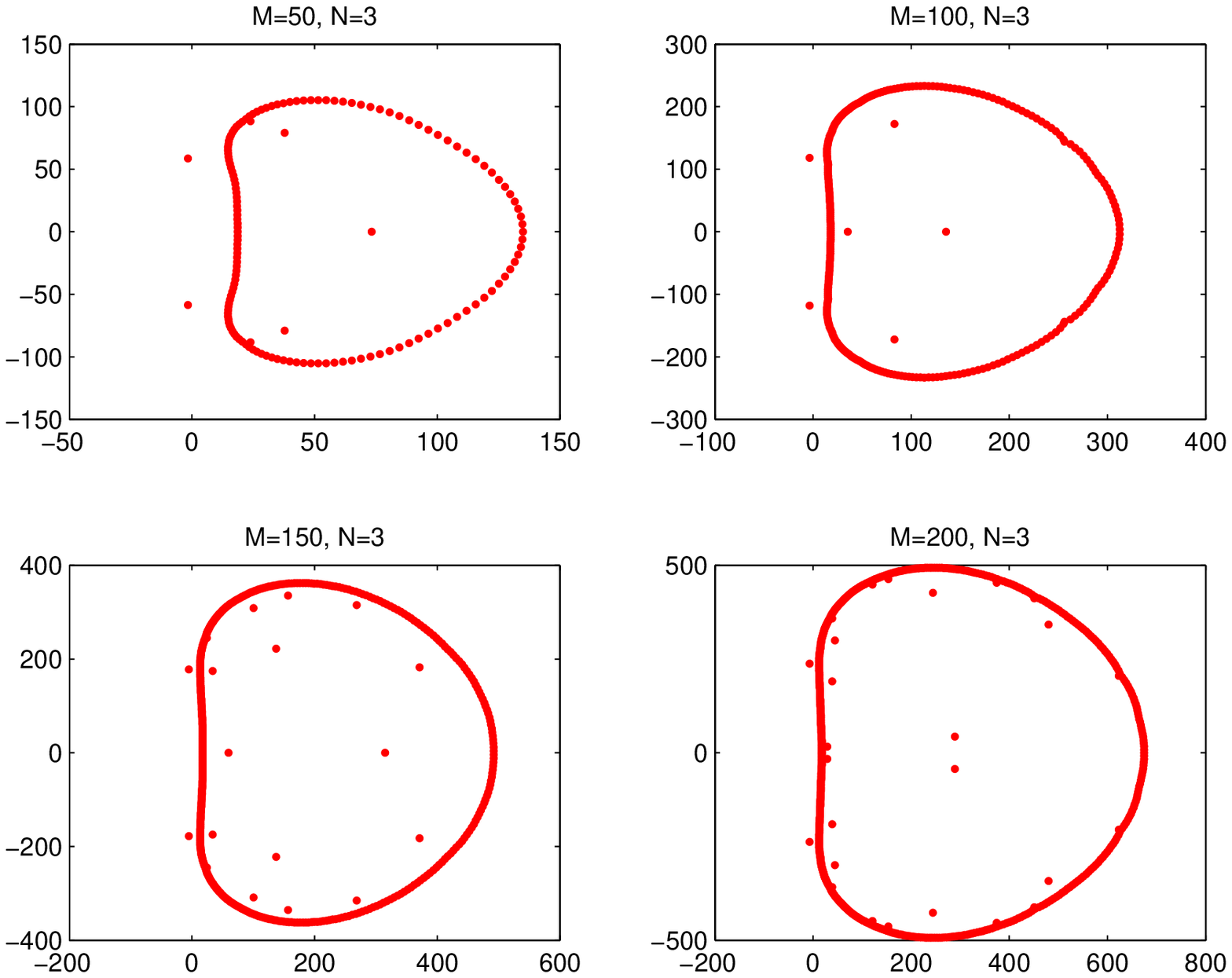}
\quad \;
\includegraphics[width=0.46\textwidth,height=0.4\textwidth]{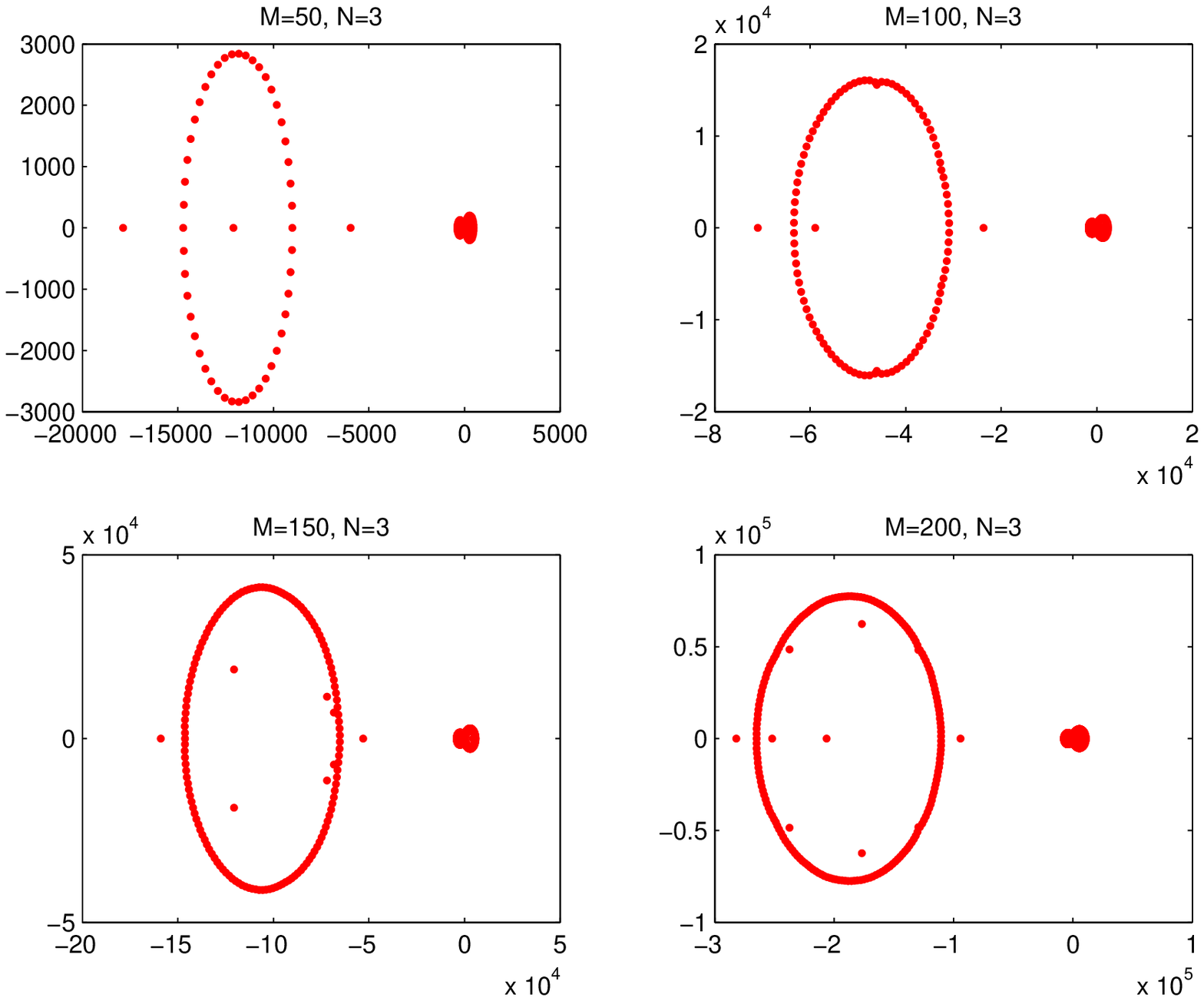}
\caption{Eigenvalues of MDFDM $\mathbf{D}^{\alf}$  for $N = 3$ and different values of $M$ with uniform mesh. Left: $\alf=1.01$, right: $\alf=1.99$.}
\label{fig4}
\end{figure}



Therefore, to overcome this issue, we minimize the jump in the integer fluxes at the interfaces by using the penalty technique. In particular, by introducing the following penalty term
\begin{equation}\label{penalty-term}
R(x)=\tau\left[\frac{du_\N}{dx}(x^+)-\frac{du_\N}{dx}(x^-)\right], \quad x\in(x_L,x_R),
\end{equation}
where $\tau$ is a penalty parameter, we obtain a modified differentiation matrix $\mathbf{D}^{\alf}+\mathbf{R}$, where $\mathbf{R}$ is given by
$$\mathbf{R}=
\begin{bmatrix}
   &  &   &  &  \\
   &   &  &  &  \\
 & &  \mathbf{0}& &  \\
   &   &   &     &  \\
\mathbf{R}^{1}  & \mathbf{R}^{2} & \cdots  &  \mathbf{R}^{\M} & \widetilde{\mathbf{R}}
\end{bmatrix}$$
with
\begin{align*}
\mathbf{R}^{i}=&
\tau(x_m)\left[\frac{d\psi_n^i(x_m^+)}{dx}-\frac{d\psi_n^i(x_m^-)}{dx}\right]_{m=1,..,M-1;n=1,\ldots,N_i-1}, i=1,\ldots,M, \\
\widetilde{\mathbf{R}}=&\tau(x_m)\left[\frac{d\phi_n(x_m^+)}{dx}-\frac{d\phi_n(x_m^-)}{dx}\right]_{m,n=1\ldots,M-1}.
\end{align*}

It is noted that the $(M-1)\times(M-1)$ matix $\widetilde{\mathbf{R}}$ is tridiagonal, the $(M-1)\times(N_1-1)$ matrix
$\mathbf{R}^{1}$ has only the first row with non-zero entries, the $(M-1)\times(N_\M-1)$ matrix
$\mathbf{R}^{\M}$ has only the last row with non-zero entries, and  the $(M-1)\times(N_i-1)$ matrix
$\mathbf{R}^{i},\,i=2,\ldots,M-1,$ has only two, i.e., the $(i-1)$-th and $i$-th rows with non-zero entries.

\begin{figure}[!t]
\centering
\includegraphics[width=0.46\textwidth,height=0.4\textwidth]{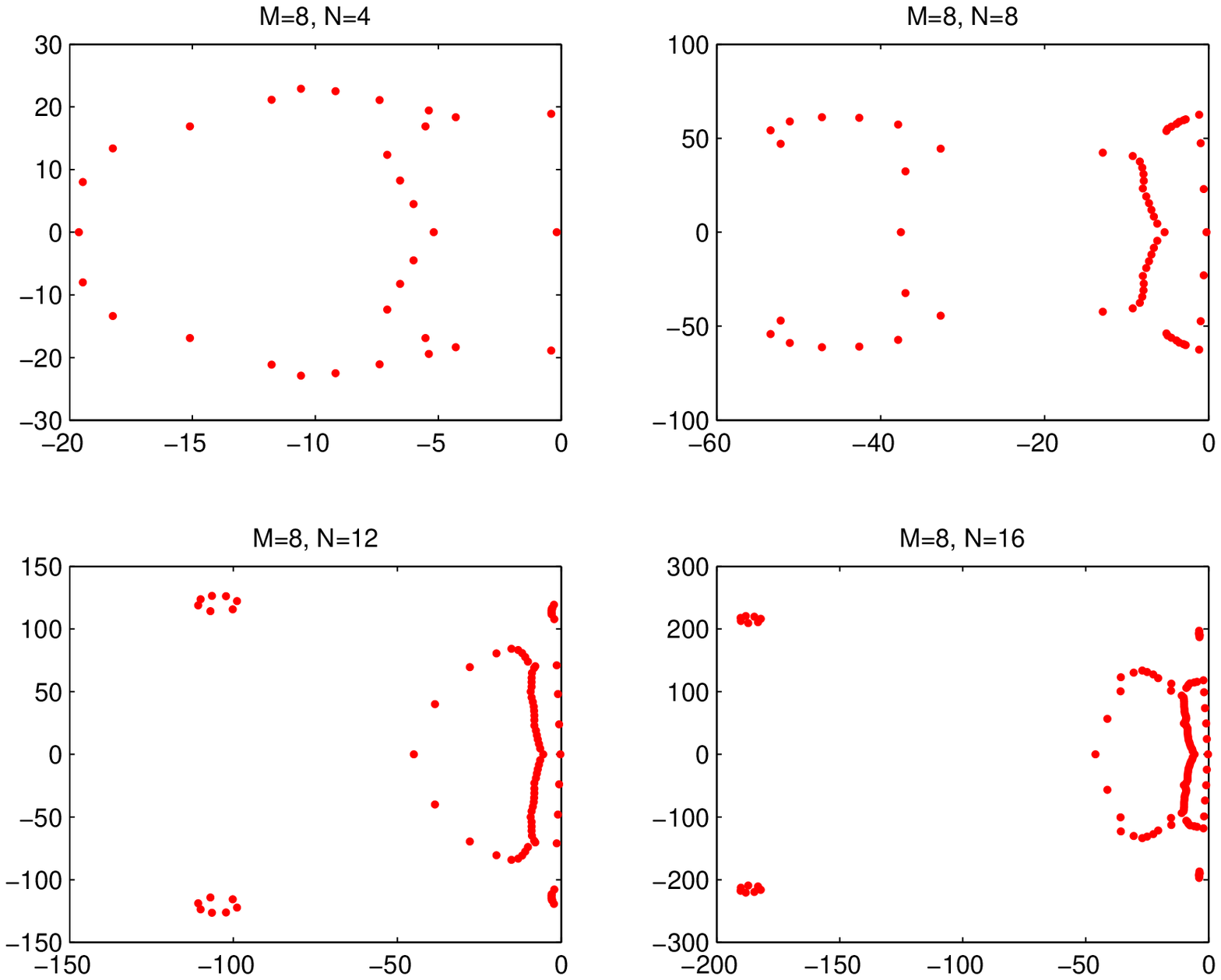}
\quad \;
\includegraphics[width=0.46\textwidth,height=0.4\textwidth]{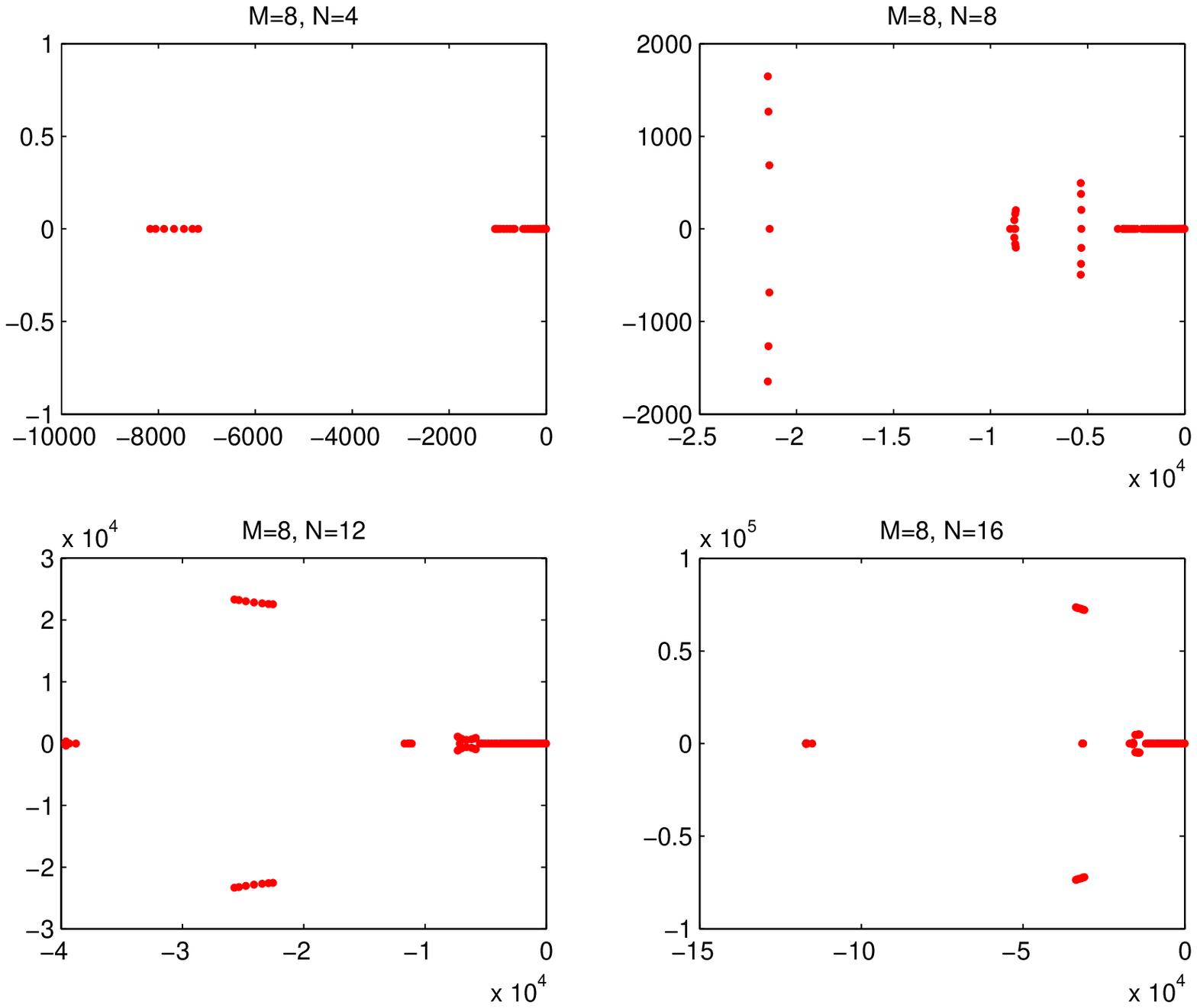}
\caption{Eigenvalues of MDFDM with penalty term, i.e., $\mathbf{D}^{\alf}+\mathbf{R}$, for $M = 8$ and different values of $N$ by using the uniform mesh. Left: $\alf=1.01,\; \tau = 1$, right: $\alf=1.99,\; \tau = 100$.}
\label{fig5}
\end{figure}

\begin{figure}[!t]
\centering
\includegraphics[width=0.46\textwidth,height=0.4\textwidth]{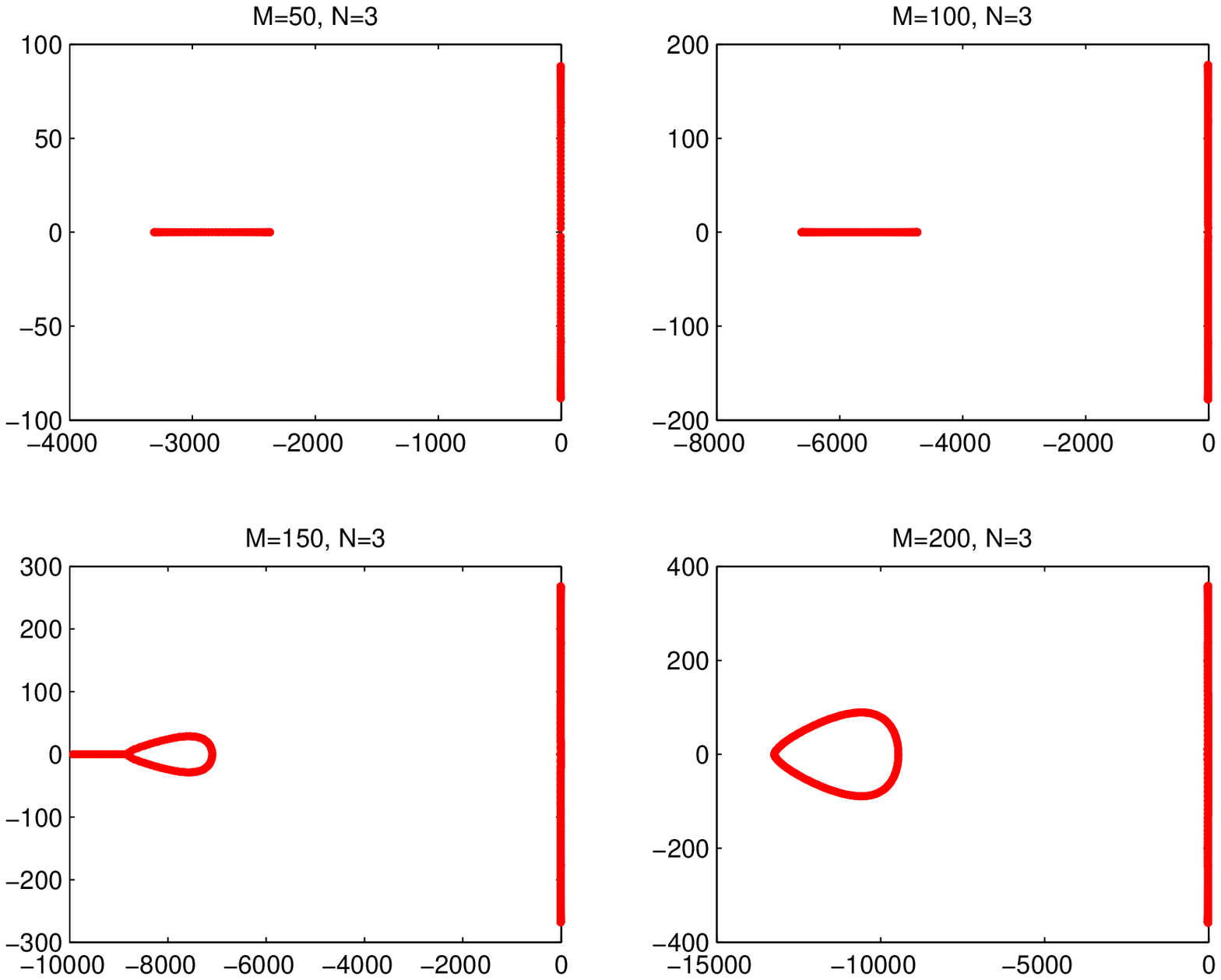}
\quad \;
\includegraphics[width=0.46\textwidth,height=0.4\textwidth]{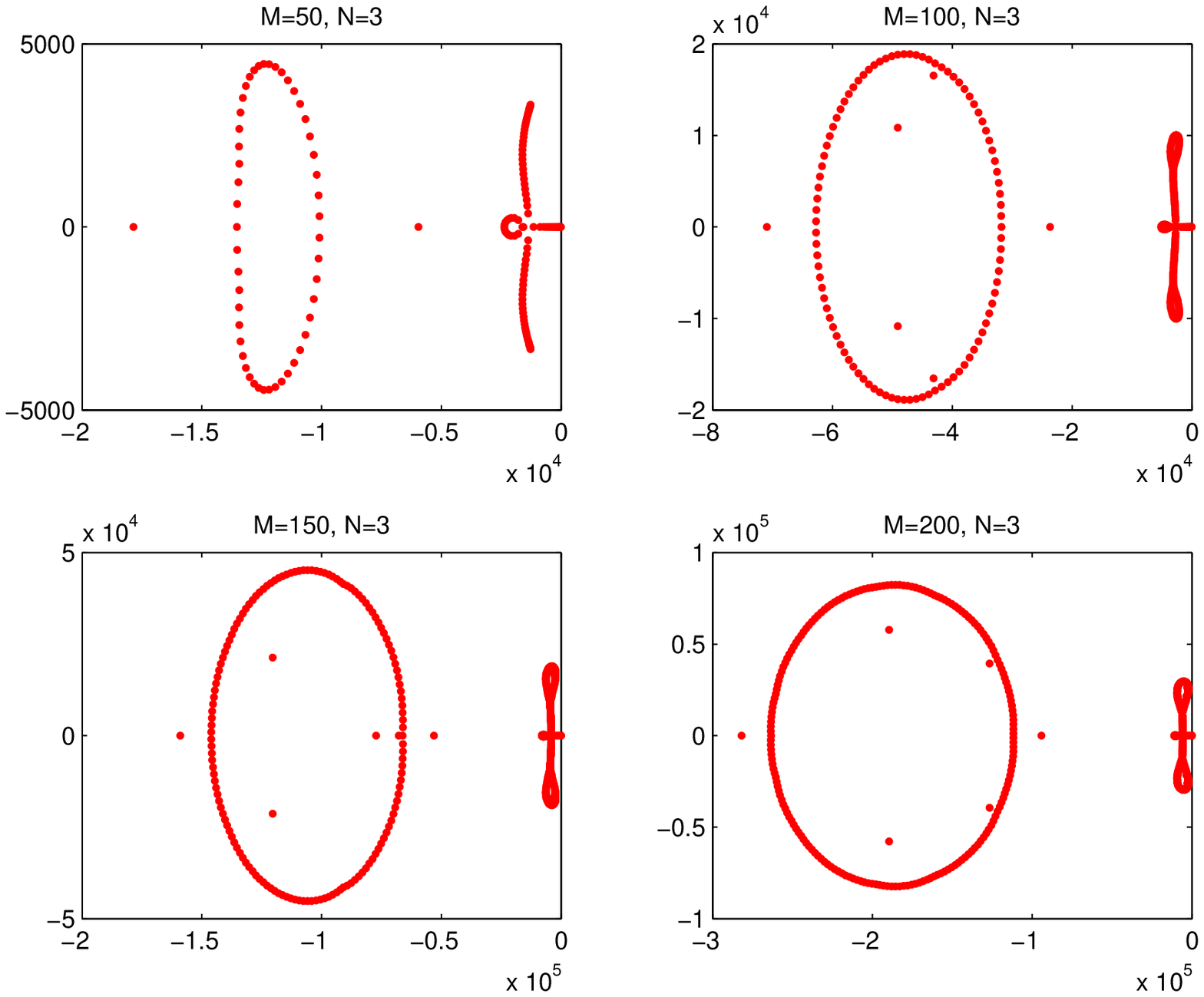}
\caption{Eigenvalues of MDFDM with penalty, i.e., $\mathbf{D}^{\alf}+\mathbf{R}$, for $N = 3$ and different values of $M$ by using the uniform mesh. Left: $\alf=1.01,\; \tau = 10$, right: $\alf=1.99,\; \tau = 10$.}
\label{fig6}
\end{figure}

%
%

\begin{figure}[H]
\centering
\includegraphics[width=8.0cm]{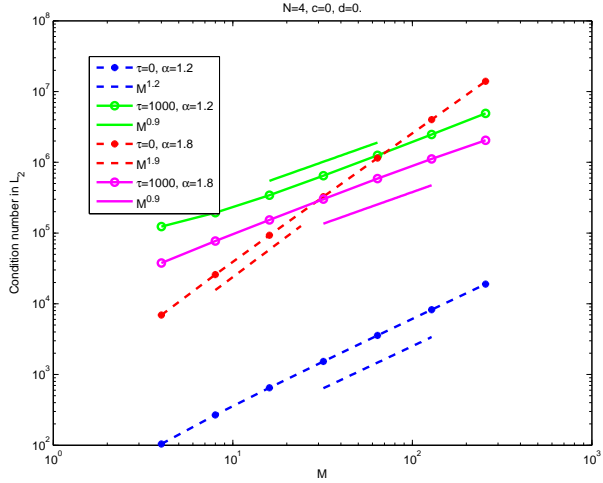}
\caption{Condition number with respect to $M$ with ($\tau=1000$) or without ($\tau=0$) penalty. The uniform mesh is used.}
\label{fig7}
\end{figure}

We show the distribution of the eigenvalues of the modified MDFDM, i.e., $\mathbf{D}^{\alf}+\mathbf{R}$, in Figs. \ref{fig5}-\ref{fig6} with the same parameters used for the MDFDM. We observe that the real parts of all eigenvalues are negative in all cases by choosing suitable penalty parameters.
The condition numbers are also investigated for some cases. In Fig. \ref{fig7}  we illustrate the condition numbers of the MDFDM in the $L_2$ norm.  We see that the penalty term can slow down the growth of the condition number.

\section{Application to the fractional Helmholtz equation and the fractional Burgers equation}\label{sec:num}
We are now in the position of numerical tests.
In the numerical tests, we use the following three types of mesh:
\begin{itemize}
\item Mesh 1: Uniform mesh:  $x_j=x_L+\frac{(x_R-x_L)j}{M},~ j=0,\ldots,M.$
\item Mesh 2: Graded mesh:   $x_j=x_L+(x_R-x_L)\left(\frac{j}{M}\right)^q,\quad  q>1, \quad j=0,\ldots,M.$
\item Mesh 3: Geometric mesh: $x_0=x_L, x_{j}=x_L+(x_R-x_L)*q^{\M-j},\quad 0<q<1,\quad j=1,\ldots,M. $
\end{itemize}
The number $\{N_i\}_{i=1}^\M$ of collocation points in each element is the same for all three cases for simplification.

\subsection{Fractional Helmholtz equation}
Let $\Lambda : = (x_L, x_R)$ and $1<\alf(x)<2$.
In this subsection we apply  the MDSCM to the following \emph{variable-order} fractional Helmholtz equation
\begin{equation}\label{mod-1}
   \lambda^2 u(x) - {_{x_L}}D^{\alf(x)}_x u(x)=f(x),\quad x\in \Lambda, \quad u(x_L)=u_L,\;  u(x_R)=u_R.
\end{equation}
The stabilized MDSCM  for \eqref{mod-1} is to find
$u_\N\in \mathbb{V}_\N$, such that
\begin{equation}\label{sd-fpc-pen}
\left[\lambda^2 u_\N-
{_{x_L}}D^{\alf(x)}_x u_\N - R(x)\right](x)=f(x),~ \forall \, x\in{\mathbb{N}}_o\setminus\{x_L,x_R\},
\end{equation}
and
\begin{equation}\label{bound-c2-pen}
u_\N(x_L)=u_L,\quad  u_\N(x_R)=u_R.
\end{equation}
The above two equations lead to the following linear system
\begin{equation}\label{sche1-mf-pen}
\left(\lambda^2 \mathbb{I}- \mathbf{D}^{\alf}-\mathbf{R}\right) \mathbf{u}=\mathbf{f}-\mathbf{r},
\end{equation}
where $\mathbb{I}$ is the unitary matrix and
\begin{align*}
\mathbf{x}&=[x_1^1,\cdots,x_{\N_1-1}^1,x_1^2,\cdots,x_{\N_2-1}^2,\cdots,x_1^\M,\cdots,x_{\N_\M-1}^\M, x_1,\cdots, x_{\M-1}]^T,\\
\mathbf{f}&=f(\mathbf{x})-u_L\left(\lambda^2\phi_0(\mathbf{x})-{_{x_L}}D^{\alpha(x)}_x\phi_0(\mathbf{x})\right)-u_R\left(\lambda^2\phi_\M(\mathbf{x})-{_{x_L}}D^{\alpha(x)}_x\phi_\M(\mathbf{x})\right),\\
\mathbf{r}&=\tau(\mathbf{x})\left\{ u_L\left[\frac{d\phi_0(\mathbf{x}^-)}{dx}-\frac{d\phi_0(\mathbf{x}^+)}{dx}\right]+
u_R\left[\frac{d\phi_\M(\mathbf{x}^-)}{dx}-\frac{d\phi_\M(\mathbf{x}^+)}{dx}\right]\right\}.
\end{align*}
and $\mathbf{u}=u_\N(\mathbf{x})$.

\begin{example}\label{exp:FHE:smu}
The first test of MDSCM is to consider the problem (\ref{mod-1})
with $[x_L,x_R]=[-1, 1]$.  The exact solution is taken as $u(x)=\sin(\pi x)$, so the homogeneous boundary conditions are implemented.
The term  $ {_{x_L}}D_x^{\alf(x)} u(x)$ is approximated by
$${_{-1}}D_x^{\alf(x)} \sin(\pi x)\approx\sum_{k=0}^L (-1)^{k+1}\frac{\pi^{2k+1}(x+1)^{2k+1-\alf(x)}}{\Gamma(2k+2-\alf(x))},$$
with $L=50$ to compute the right hand function (RHF) $f(x)$. For $\alf(x)$,
we consider the following two cases:
\begin{enumerate}
\item The constant-order $\alf=1.1,1.5,1.9$.
\item  The variable-order $\alf(x)=1.1+\frac{x+1}{2.5}$.
\end{enumerate}
\end{example}

The aim of this example is to test the accuracy of the proposed method for the smooth solution. In this example the uniform mesh is used.
The maximum errors between the numerical solution and exact solution are shown in Figs. \ref{fig8}-\ref{fig9}.
We observe from the left plots of Fig. \ref{fig8} and Fig. \ref{fig9} that the spectral accuracy is obtained for both constant-order and variable-order fractional derivative. Also, the good accuracy can be obtained by $h$-refinement (right plots of Fig. \ref{fig8} and Fig. \ref{fig9}).
We also observe from Figs. \ref{fig8}-\ref{fig9} that the accuracy can be improved significantly by the penalty method.

\begin{figure}[!t]
\centering
\includegraphics[width=0.48\textwidth,height=0.4\textwidth]{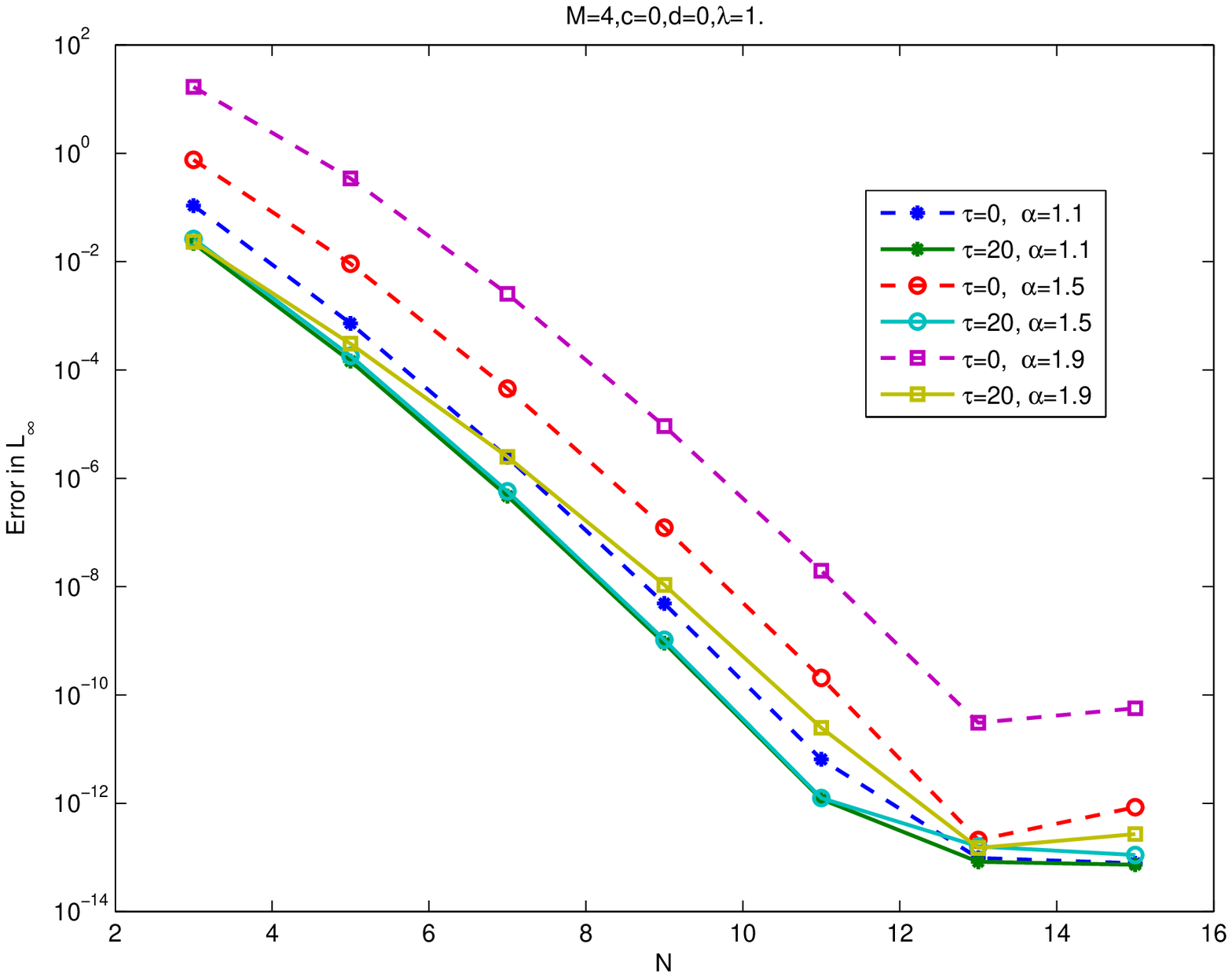}
\includegraphics[width=0.48\textwidth,height=0.4\textwidth]{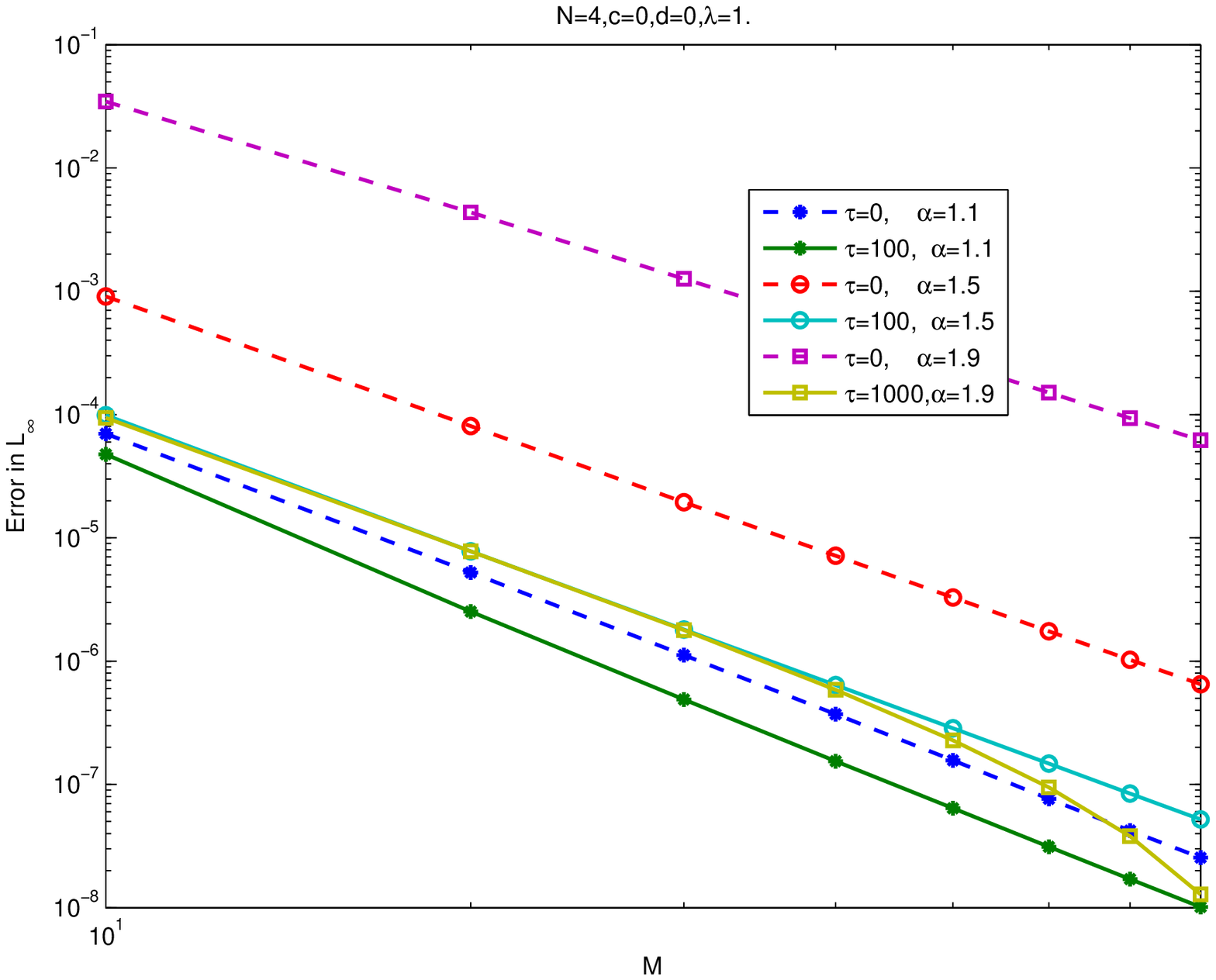}
\caption{$L_\infty$-error for Example \ref{exp:FHE:smu} with the uniform mesh (Mesh 1). Left: $p$-refinement ($M = 4$), right: $h$-refinement ($N =4$). }
\label{fig8}
\end{figure}

%

\begin{figure}[!t]
\centering
\includegraphics[width=0.48\textwidth,height=0.4\textwidth]{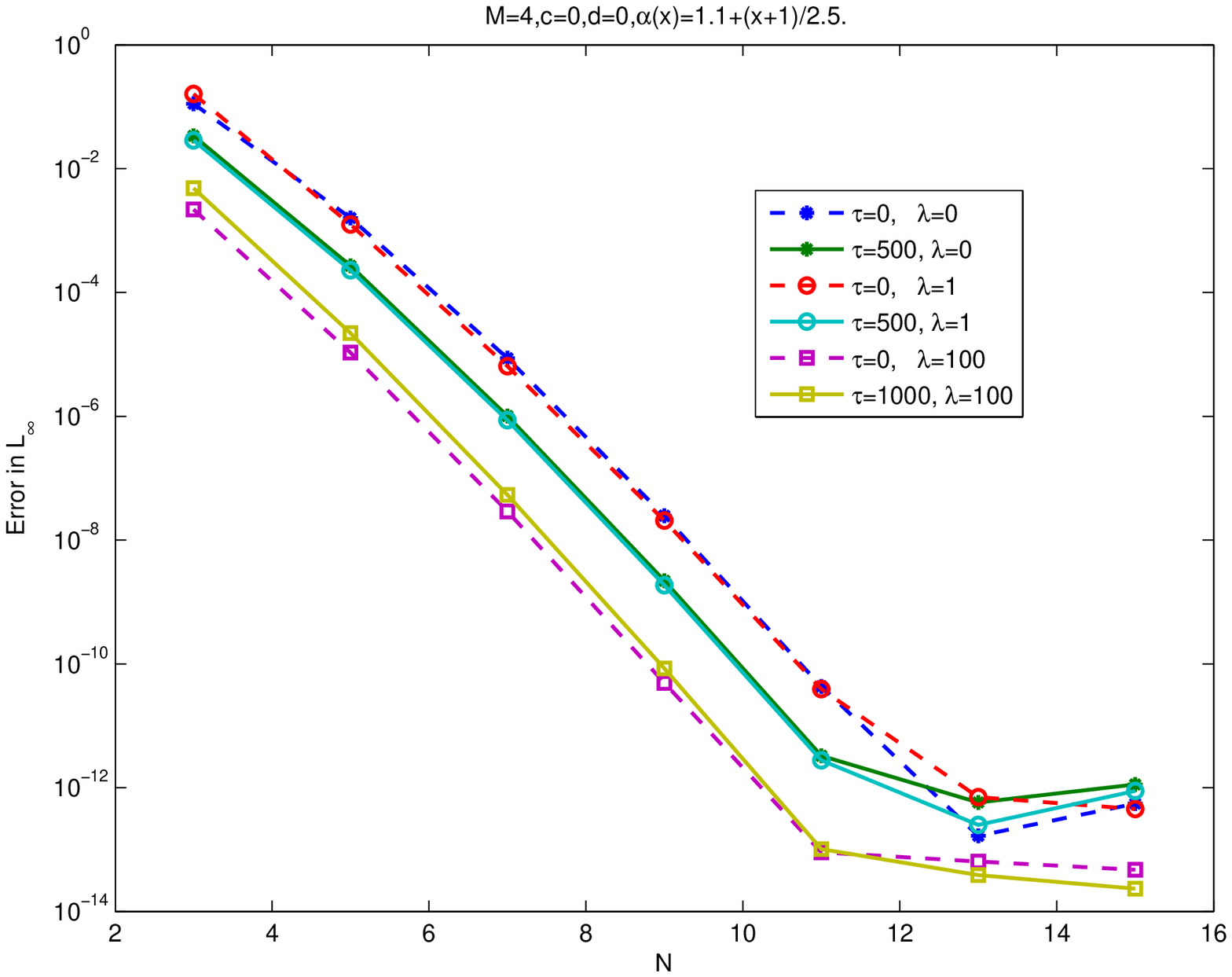}
\includegraphics[width=0.48\textwidth,height=0.4\textwidth]{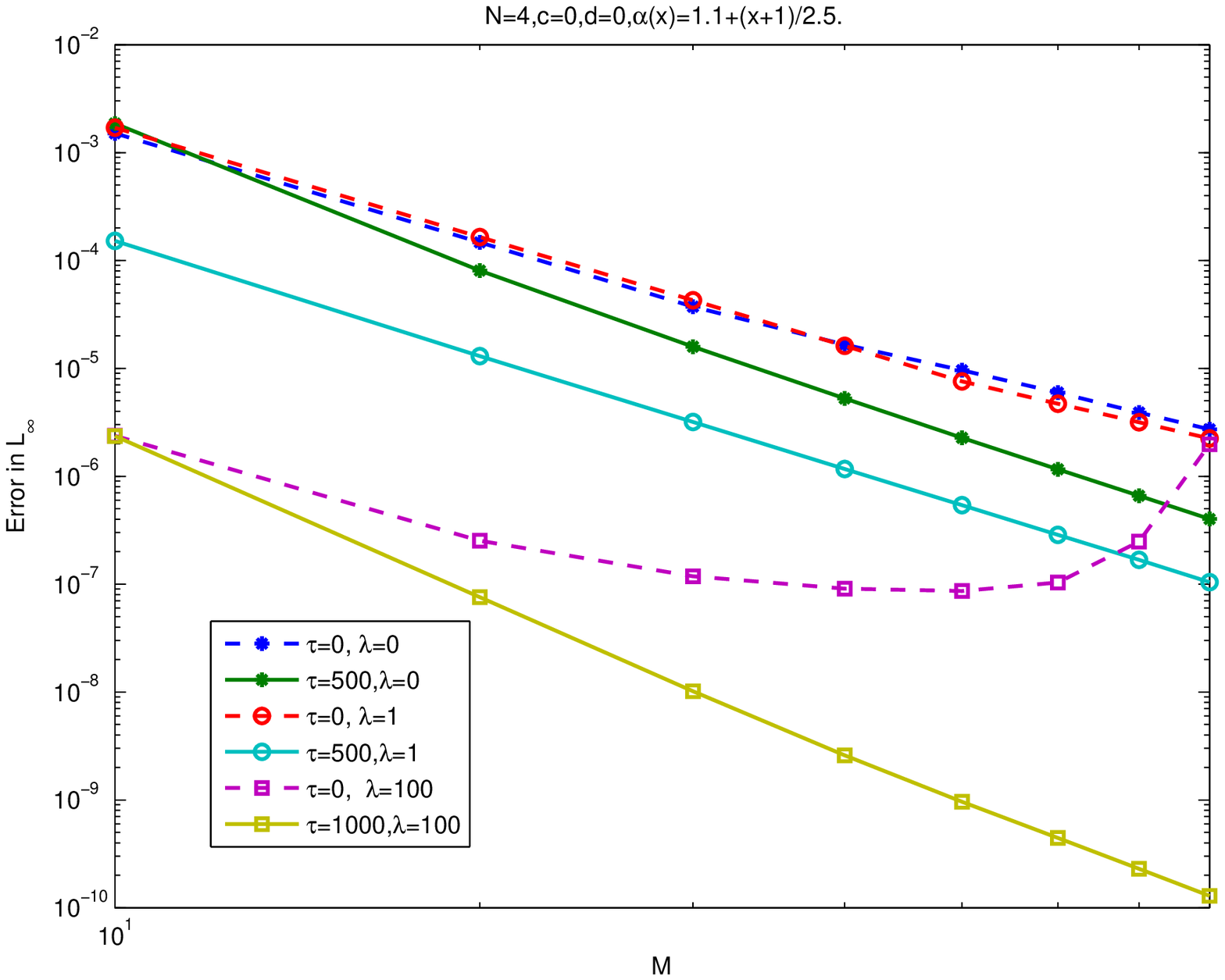}
\caption{Error in $L_\infty$ for Example \ref{exp:FHE:smu} with the uniform mesh (Mesh 1). Left: $p$-refinement ($M = 4$), right: $h$-refinement ($N = 4$). }
\label{fig9}
\end{figure}


\begin{example}\label{exp:FHE:smf}
In the second test of the MDSCM we also consider the model problem (\ref{mod-1})
with $[x_L,x_R]=[-1, 1]$.  The exact solution is taken as $u(x)=(1-x)(1+x)^{\alf(x)-1}$, and then we can obtain the RHF
$$f(x)=\lambda^2 u(x)+\Gamma(1+\alf(x)).$$
When $\lambda=0$, we obtain a smooth RHF $f$ if the order $\alf(x)$ is smooth.
We can see that the exact solution has very low regularity since $0<\alpha(x)-1<1$. We aim to find out whether the MDSCM
can give a good approximation for a low regularity solution.
\end{example}

The maximum errors between the numerical and exact solutions are plotted in Figs. \ref{fig10}-\ref{fig11}.
We see that the MDSCM without penalty ($\tau=0$) hardly achieves any accuracy when $\alf(x)-1\in (0,1)$, whereas the MDSCM with the penalty term ($\tau\neq0$) always converges. Also, we can observe that the geometric mesh (Mesh 3) can achieve better accuracy
than the graded mesh (Mesh 2) when $\alf(x)-1$ is close to zero. We also compare with the spectral element method (``SEM") proposed in  \cite{MaoS17}  and we see that while it is better than the MDSCM, the penalty-based version exhibits superior performance overall.

\begin{figure}[!t]
\centering
\includegraphics[width=0.48\textwidth,height=0.4\textwidth]{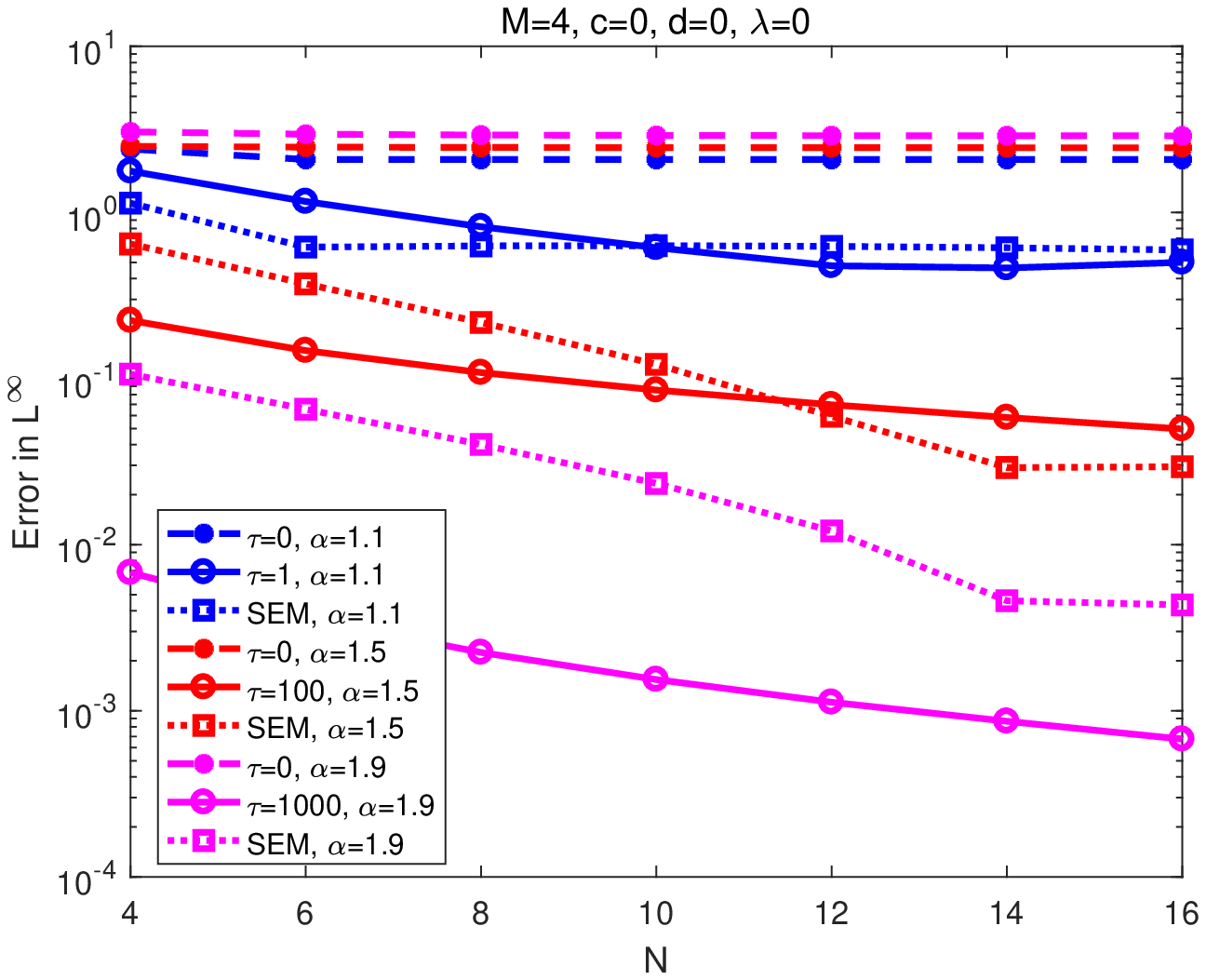}
\includegraphics[width=0.48\textwidth,height=0.4\textwidth]{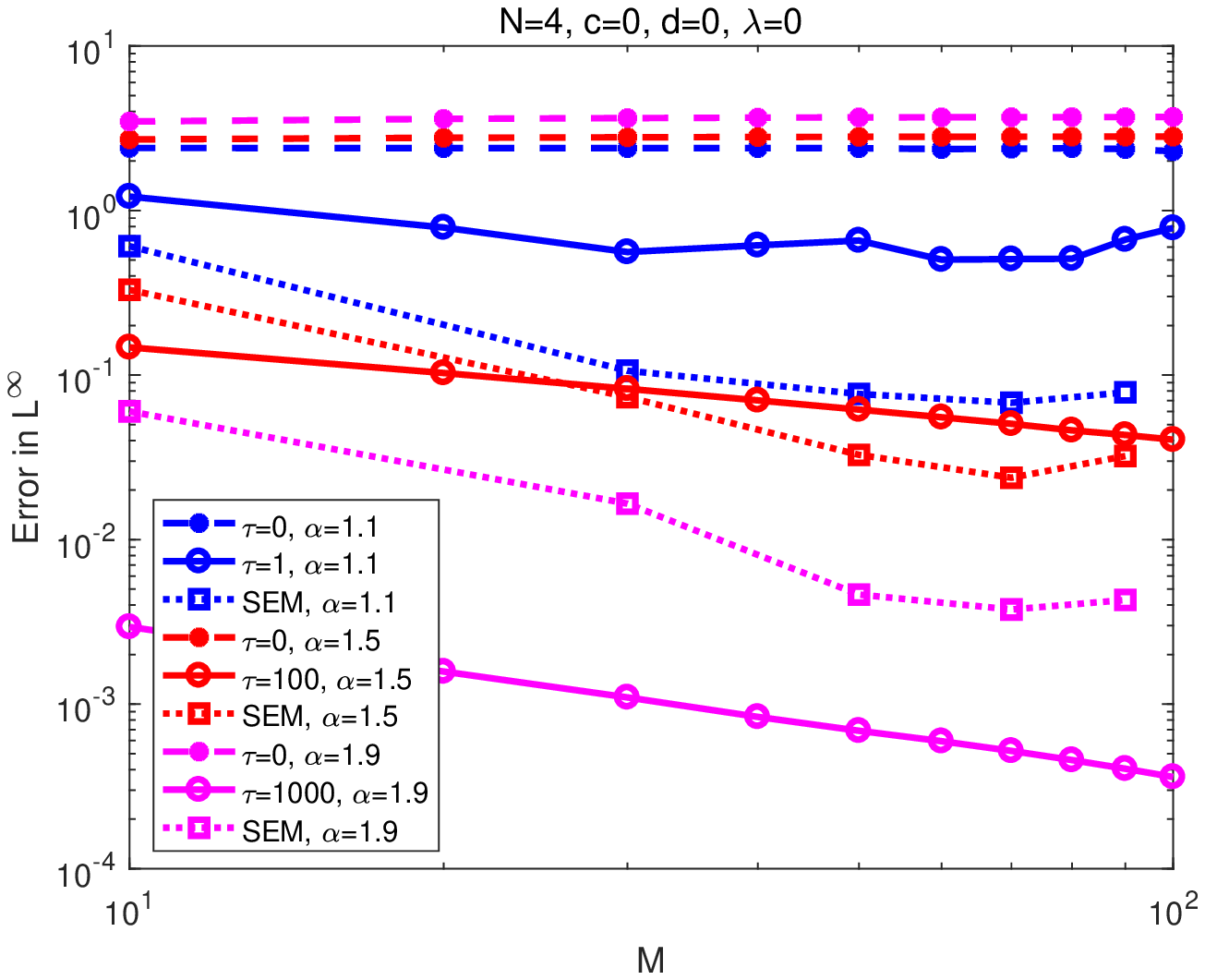}
\caption{Error in $L_\infty$ for Example \ref{exp:FHE:smf} with the uniform mesh (Mesh 1). Left: $p$-refinement ($M = 4$), right: $h$-refinement ($N = 4$). }
\label{fig10}
\end{figure}

\begin{figure}[!t]
\centering
\includegraphics[width=0.48\textwidth,height=0.4\textwidth]{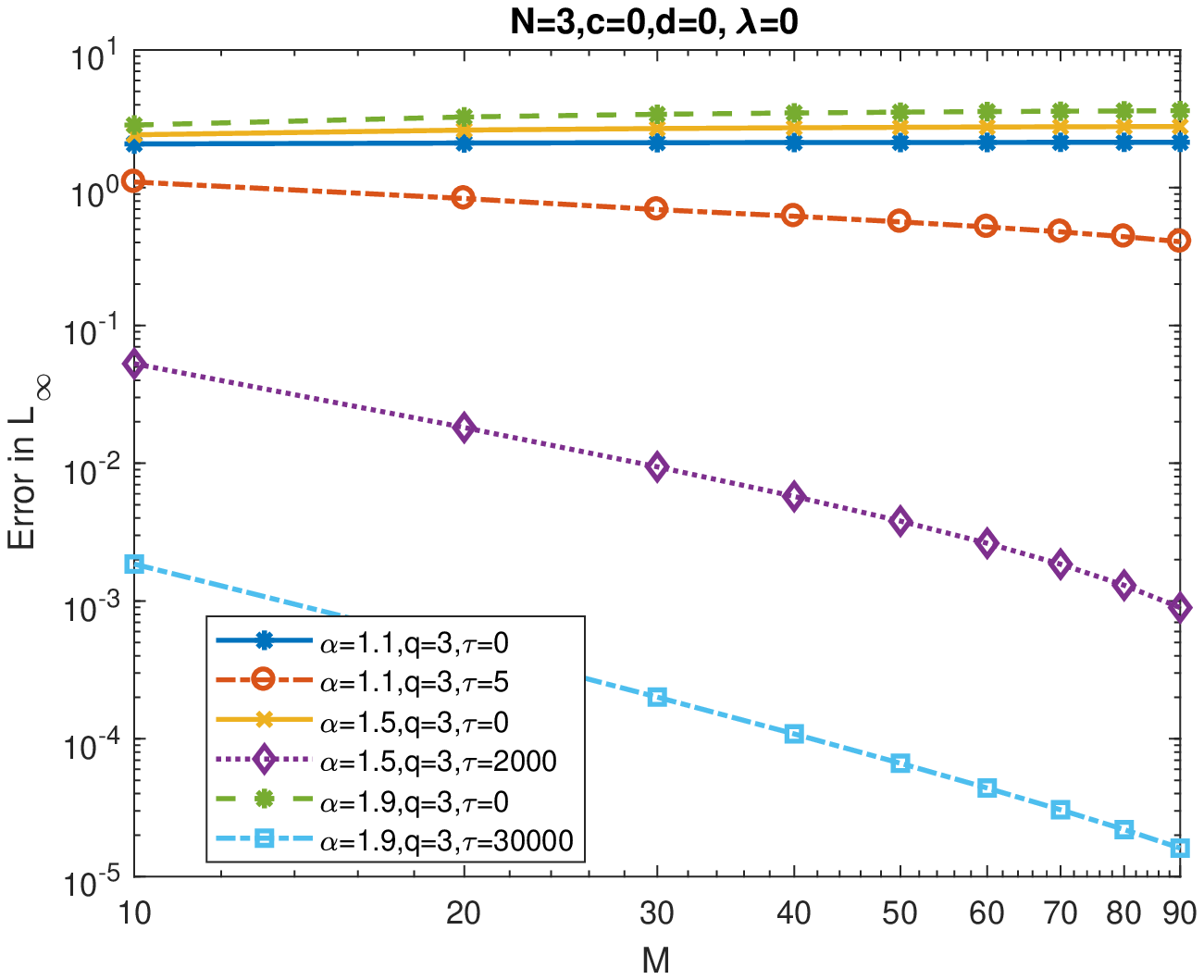}
\;
\includegraphics[width=0.48\textwidth,height=0.4\textwidth]{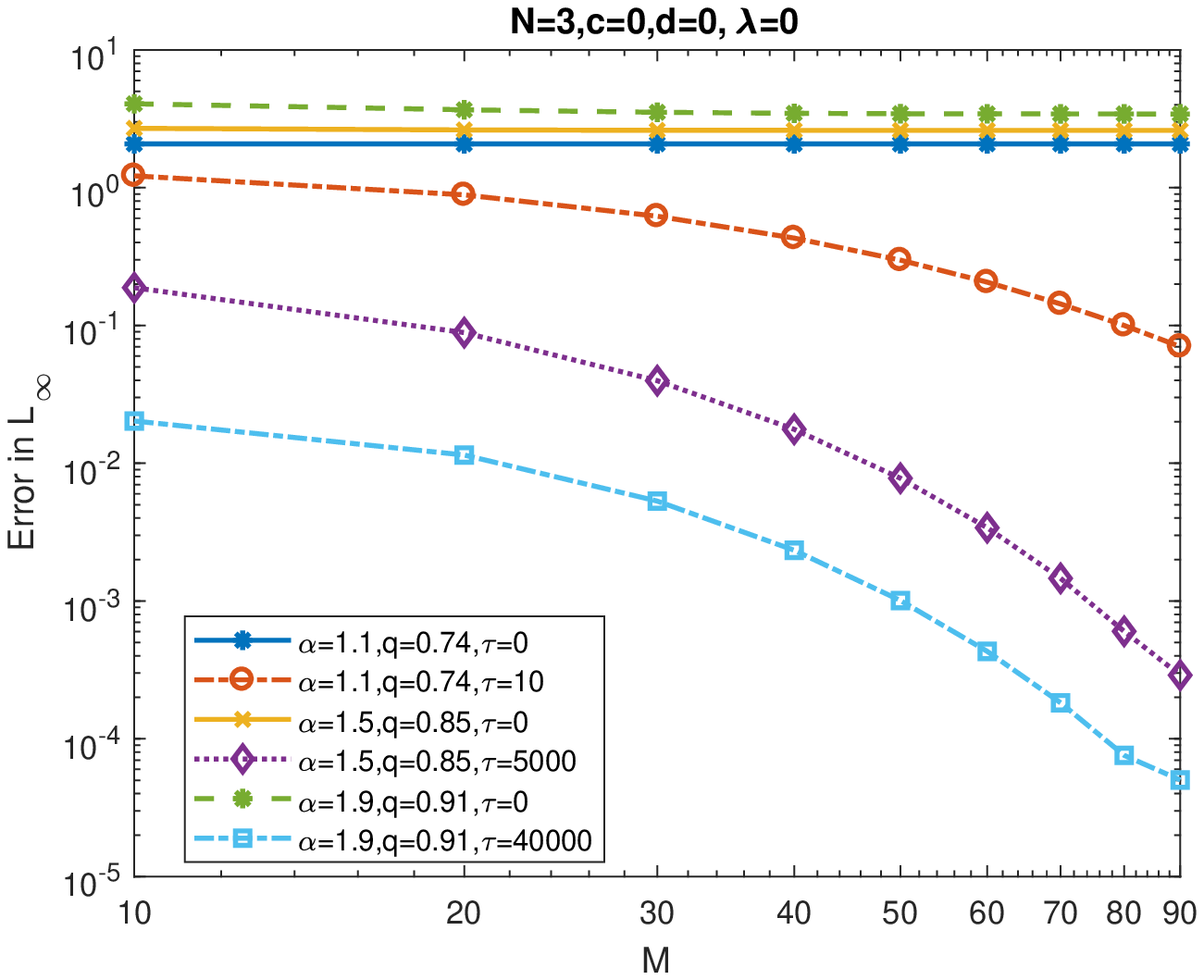}
\caption{Error in $L_\infty$ for Example \ref{exp:FHE:smf} for different values of $\alpha$. Left: graded mesh, right: geometric mesh.}
\label{fig11:0}
\end{figure}

\begin{figure}[!t]
\centering
\includegraphics[width=0.48\textwidth,height=0.4\textwidth]{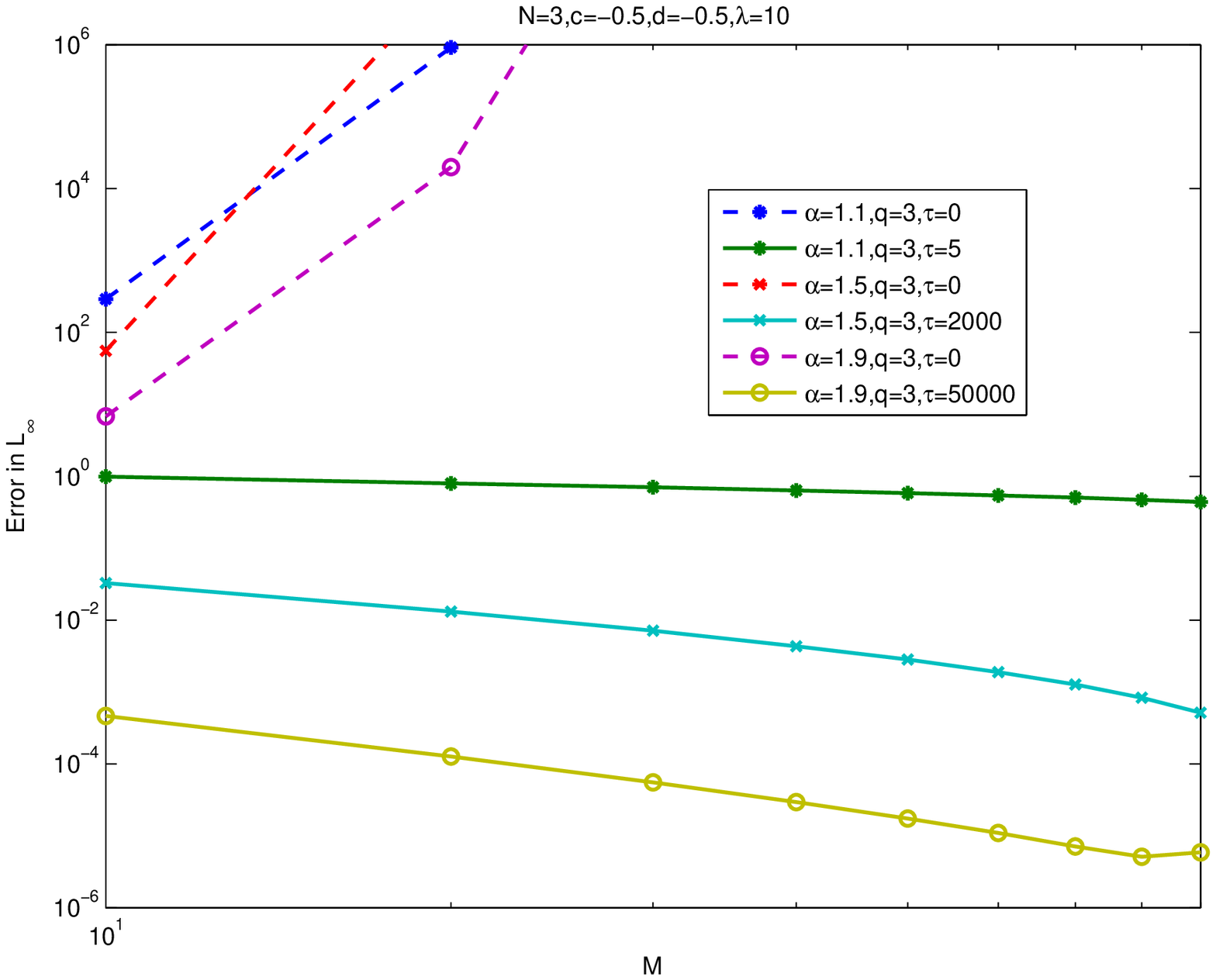}
\includegraphics[width=0.48\textwidth,height=0.4\textwidth]{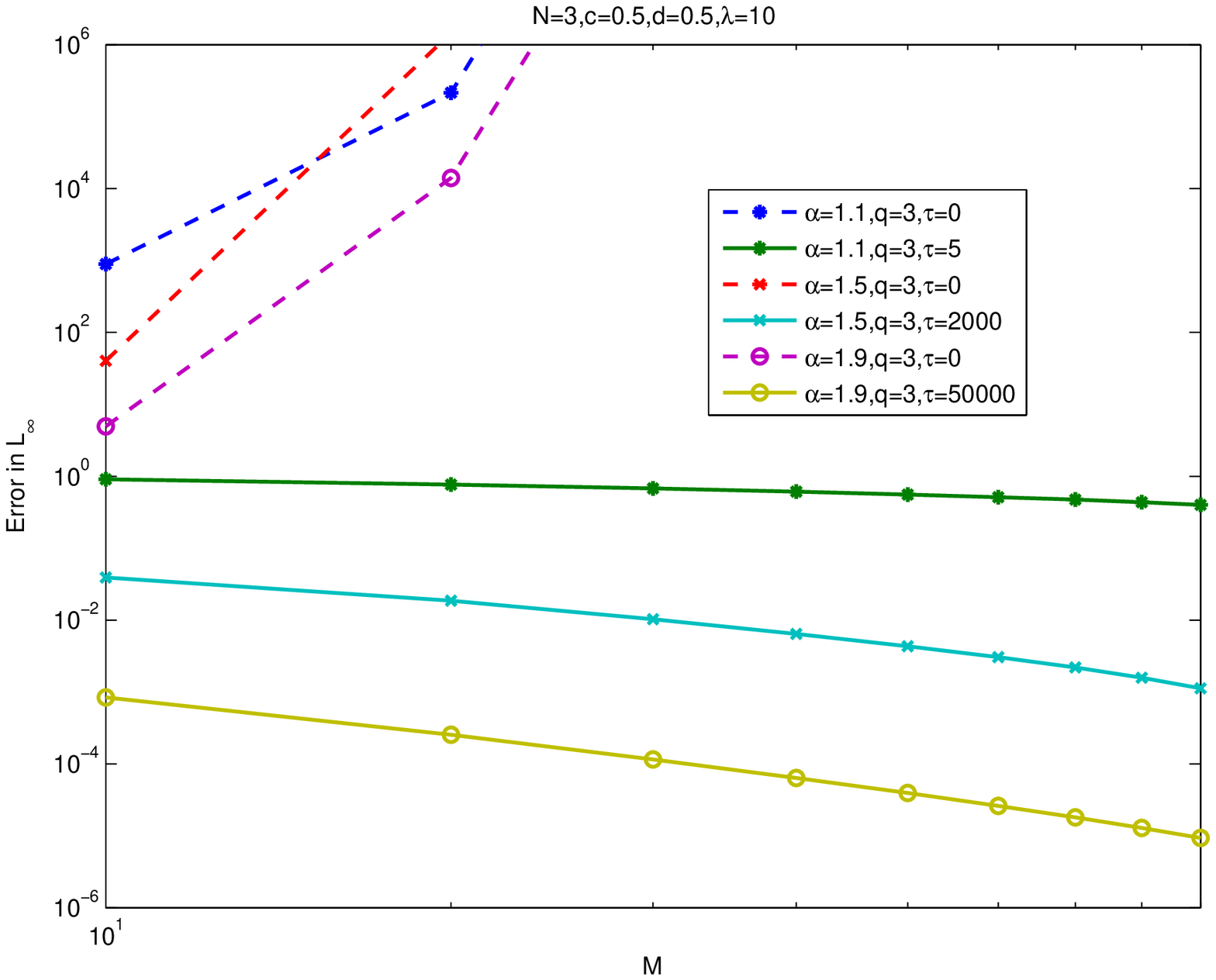}
\caption{Error in $L_\infty$ for Example \ref{exp:FHE:smf} with graded mesh and different Jacobi interpolants. Left: $c=d=-1/2$, right: $c=d=1/2$.}
\label{fig11}
\end{figure}

\subsection{Fractional Burgers equation}
In this subsection we employ the MDSCM to solve a time dependent problem, i.e., the following fractional Burgers equation (FBE)
\begin{equation}\label{examp-3}
             \partial_t u(x,t) +u(x,t) \partial_x u(x,t)=\epsilon D^{\alf(x,t)} u(x,t),
\end{equation}
subject to homogeneous Dirichlet boundary conditions and initial condition $u(x,0)=u_0(x)$, where
$\epsilon>0,\; 1<\alf(x,t)<2, \; (x,t)\in(-1,1)\times(0,1]$.

For the time discretization, we employ a semi-implicit time-discretization scheme, namely, the two-step second-order Crank-Nicolson/leapfrog scheme, then, the full discretization scheme reads as: for $n=1,2,\ldots,$
\begin{equation}\label{schm-cnmf}
\left\{
 \begin{array}{ll}
        \left(\mathbb{I}-\Delta t\epsilon (\mathbf{D}^{\alf^{n+1}}+\mathbf{R})\right) \mathbf{u}^{n+1}=\mathbf{g}, \\
          \mathbf{u}^{1}=\mathbf{u}^{0}+\Delta t\epsilon \mathbf{D}^{\alf^{0}} \mathbf{u}^{0}-\Delta t( \mbox{diag}(\mathbf{u}^0) \mathbf{D}\mathbf{u}^0) ,\\
              \mathbf{u}^0=u_0(\mathbf{x}),
             \end{array}
           \right.
\end{equation}
where
$$\mathbf{g}=\left(\mathbb{I}+\Delta t\epsilon \mathbf{D}^{\alf^{n-1}}\right) \mathbf{u}^{n-1}-2\Delta t( \mbox{diag}(\mathbf{u}^n) \mathbf{D}\mathbf{u}^n),$$
$\mathbf{D}^{\alf^{n}},\; \mathbf{R}$ and $\mathbf{D}$ are the MDFDM of $\alf$-order, the penalty matrix and
the first-order differentiation matrix, respectively.

\begin{example}\label{exp:Burgers}
In this example, we consider the initial condition $u_0(x)=\sin(\pi x)$ and the following five cases of fractional order considered in \cite{ZenZK15}:
\begin{itemize}
\item Case 1: (constant-order) $\alf(x,t)=1.1,1.2,1.3,1.5,1.8; $
\item Case 2: (monotonic increasing-order) $\alf(x,t)=1+\frac{5+4x}{10};$
\item Case 3: (monotonic decreasing-order) $\alf(x,t)=1+\frac{5-4x}{10};$
\item Case 4: (nonsmooth order) $\alf(x,t)=\frac{4}{5}|\sin(10\pi(x-t))|+1.1;$
\item Case 5: (nonsmooth order) $\alf(x,t)=\frac{4|xt|}{5}+1.1.$
\end{itemize}
\end{example}

We first consider the constant-order case, i.e., Case 1. We show the numerical solutions of the FBE \eqref{examp-3} at time $t=1$ in Fig. \ref{fig:allsolution:uniform} for different values of $\alpha$ by using the uniform mesh. A comparison of the numerical solutions for $\alpha = 1.5$ is also shown in Fig. \ref{fig:alp15} by using $h$ or $p$ refinement. Observe that the obtained numerical result is the same as the one obtained in \cite[Fig. 5]{ZenZK15} and the solutions near the left boundary have sharp transitions, especially for smaller values of $\alpha$.

\begin{figure}[!t]
\centering
\includegraphics[width=8.0cm]{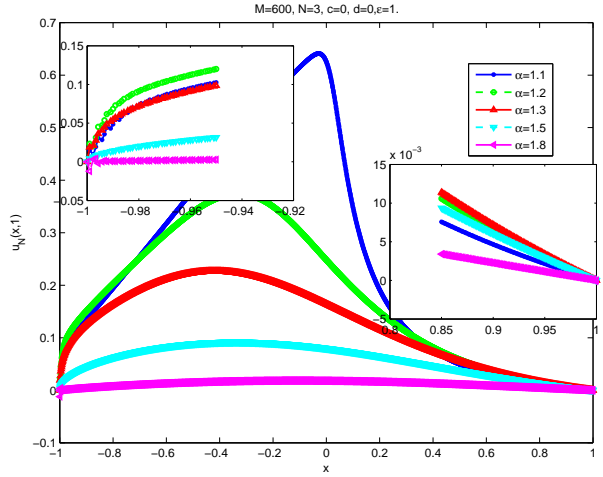}
\caption{Numerical solutions at $t=1$ for Example \ref{exp:Burgers} (Case 1) with the uniform mesh (Mesh 1). $\epsilon=1,\, M = 600,\, N=3,\, c=d=0,\, \tau=10^3,\, \Delta t=10^{-3}$. }
\label{fig:allsolution:uniform}
\end{figure}

\begin{figure}[!t]
\centering
\includegraphics[width=0.48\textwidth,height=0.4\textwidth]{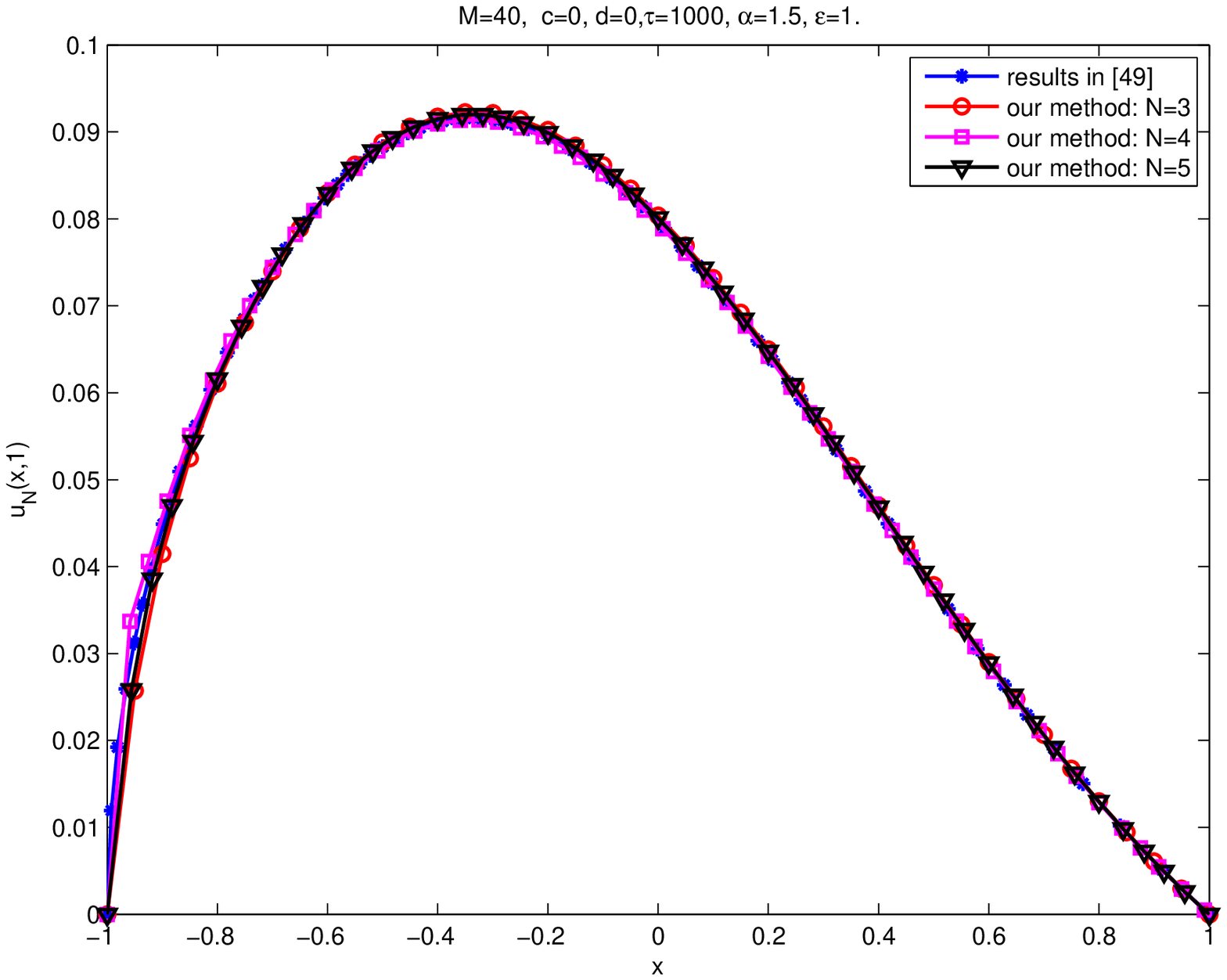}
\includegraphics[width=0.48\textwidth,height=0.4\textwidth]{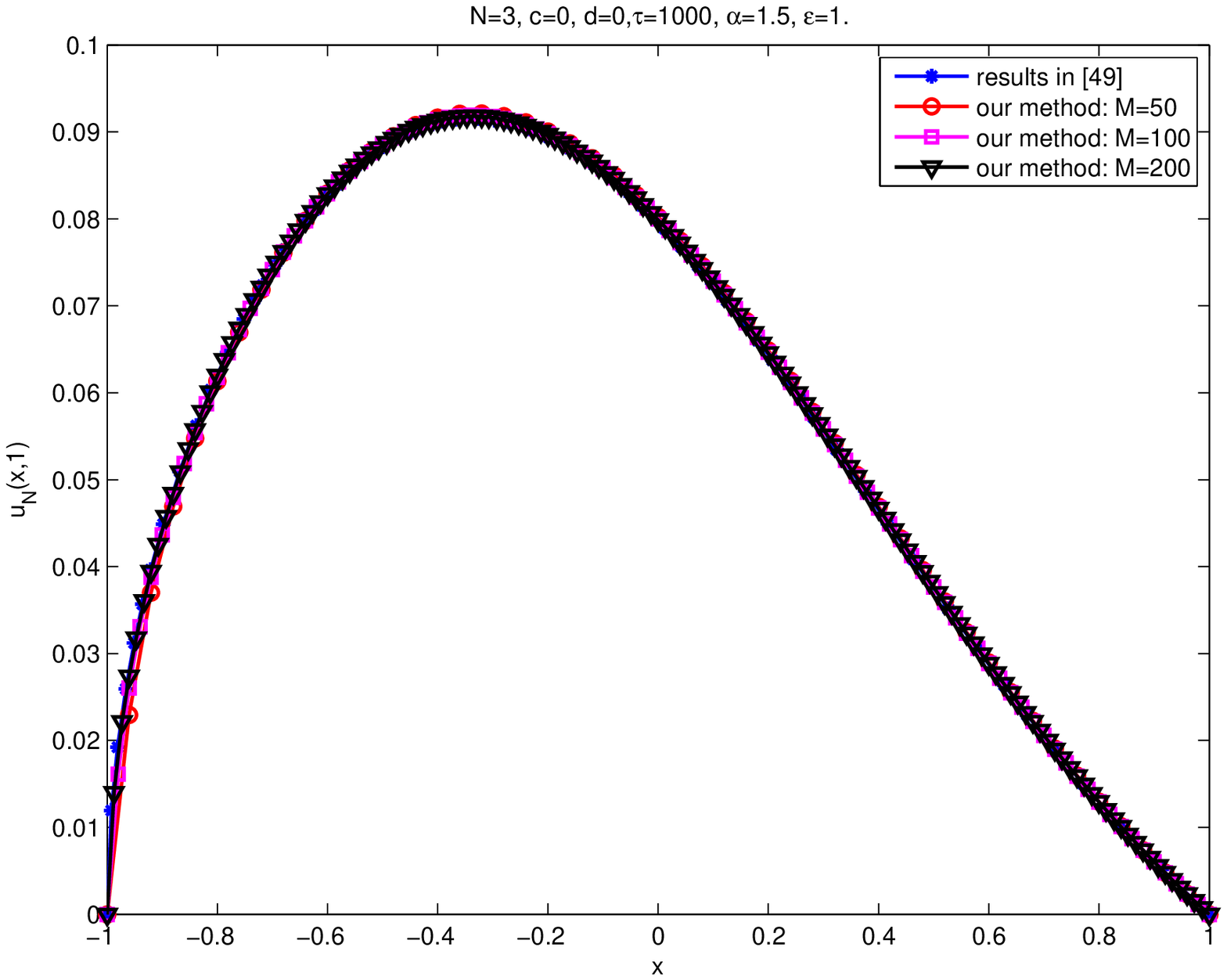}
\caption{Comparison of numerical solutions at $t=1$ for Example \ref{exp:Burgers} with the one obtained in [49] for $\alpha = 1.5$. Left: $p$-refinement, right: $h$-refinement.}
\label{fig:alp15}
\end{figure}


Furthermore, to illustrate the effect of the proposed MDSCM in resolving the issue of singularities, we show the numerical solutions for $\alpha = 1.1$ at time $t = 1$ in Fig. \ref{fig:alp11uniformgradedgeom} with uniform mesh (left plot) and graded mesh or geometric mesh (right plot), respectively. The reference solution is computed by using the graded mesh with $M = 200,\, q=3$. We observe that the one-domain spectral method, although it has better accuracy at the left boundary, it exhibits oscillations, which propagate and eventually renderer the solution erroneous. However, the $h$-refinement can resolve this issue. Moreover, by using  graded or geometric meshes, we can further enhance the accuracy of the solution. Also, the result by using the graded mesh is more accurate compared with the results obtained by using the uniform or the geometric mesh.
We point out here that for the geometric mesh, we first divide the interval $[-1,1]$ into two subintervals $[-1,-0.95]$ and $[-0.95,1]$, and subsequently use a geometric mesh for the first subinterval with 10 spectral elements and $q= 0.5$ while we use a uniform mesh for the second subinterval with $M-10$ elements, where $M$ is the total number of elements over the entire interval.

\begin{figure}[!t]
\centering
\includegraphics[width=0.48\textwidth,height=0.4\textwidth]{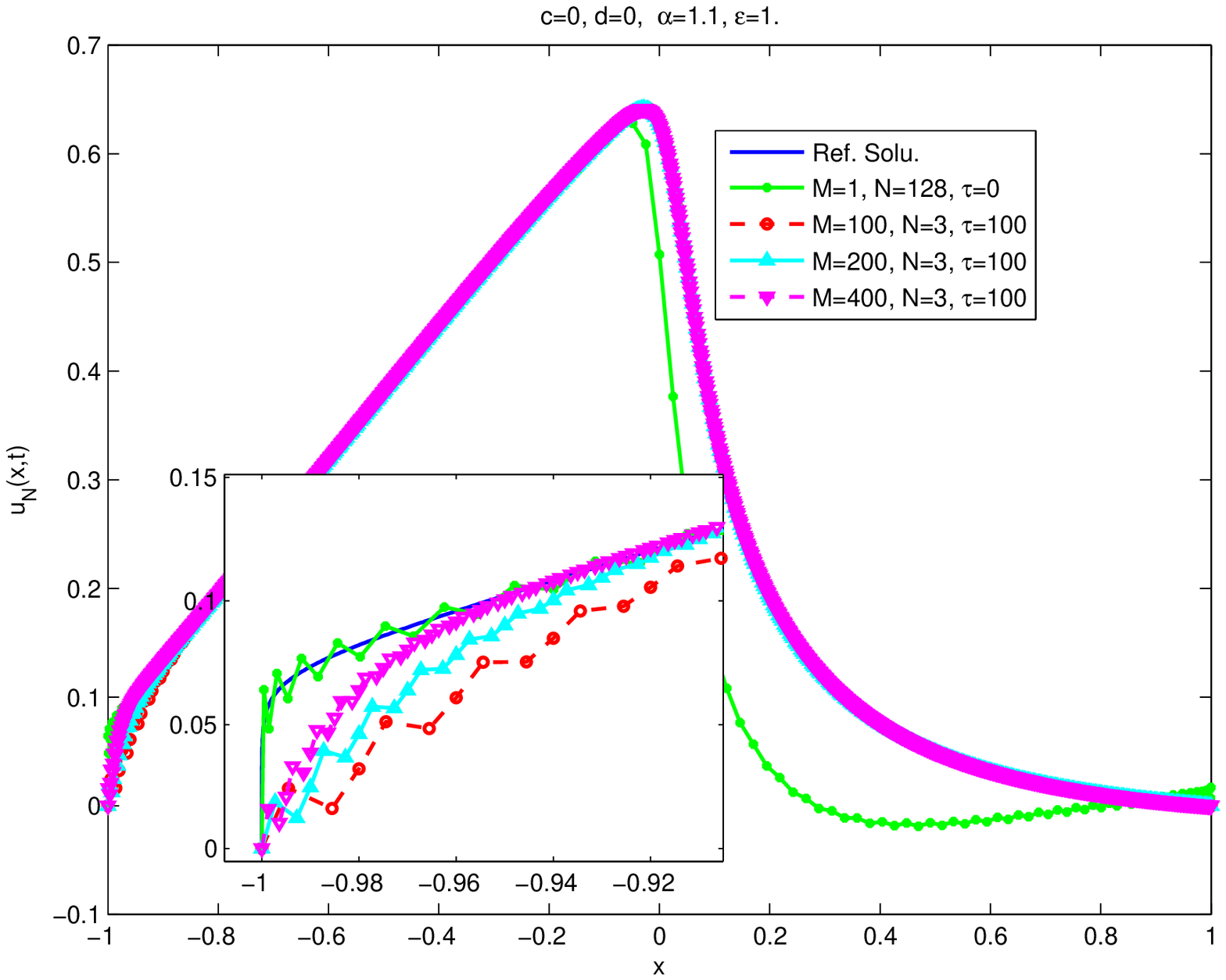}
\includegraphics[width=0.48\textwidth,height=0.4\textwidth]{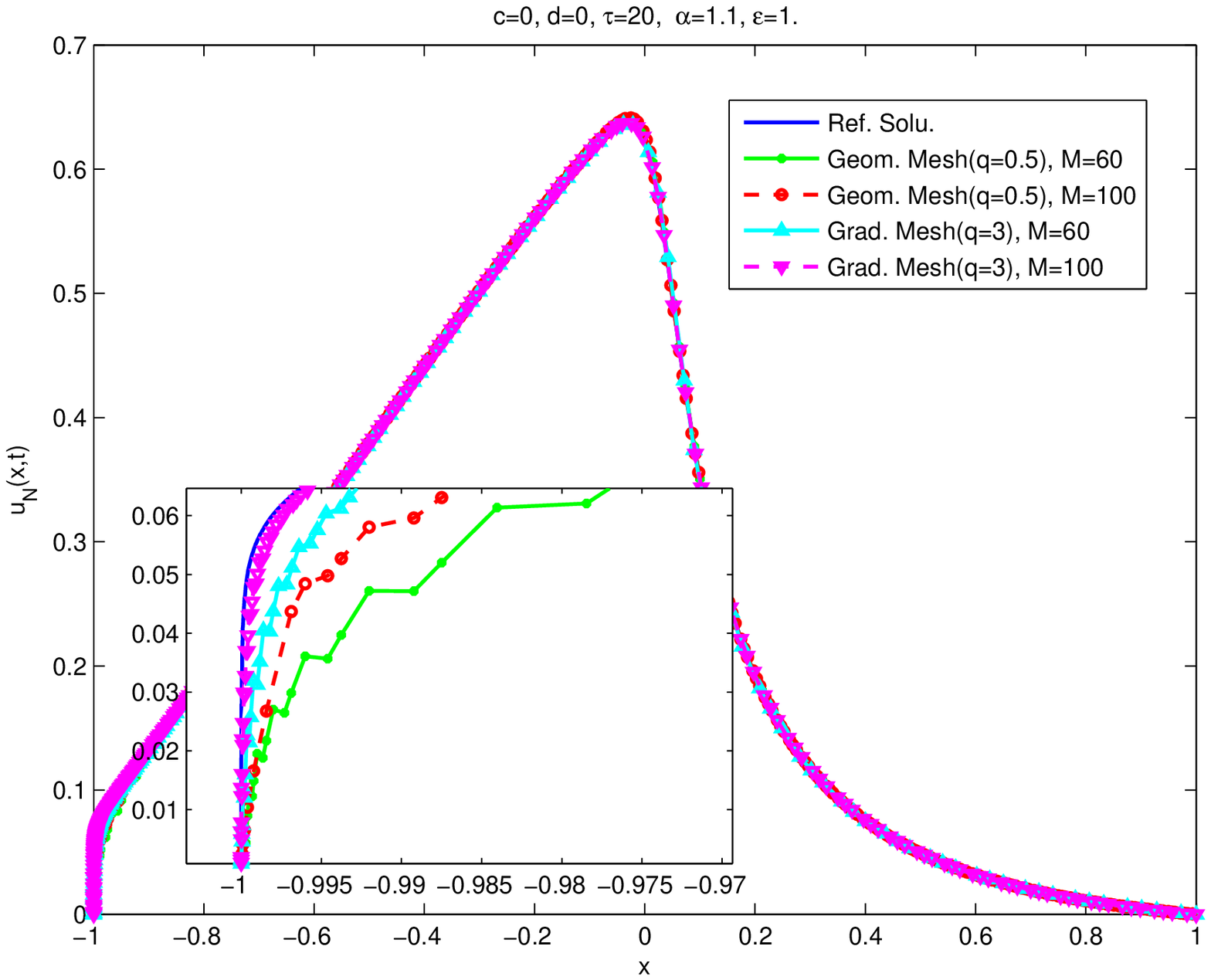}
\caption{Numerical solutions at $t=1$ for Example \ref{exp:Burgers} (Case 1) and the value of the fractional order $\alpha = 1.1$. Left: uniform mesh, right: graded mesh and geometric mesh.}
\label{fig:alp11uniformgradedgeom}
\end{figure}

%
%

We also study the behavior of the solutions with different values of viscosity $\varepsilon$. The numerical results for different values of fractional order $\alpha = 1.3, \,1.8$ are shown in Fig. \ref{fig:diffviscosity}. We observe high degree of sharpness when the value of viscosity $\varepsilon$ is very small, and this sharpness can be captured by using the proposed MDSCM.

\begin{figure}[!t]
\centering
\includegraphics[width=0.48\textwidth,height=0.4\textwidth]{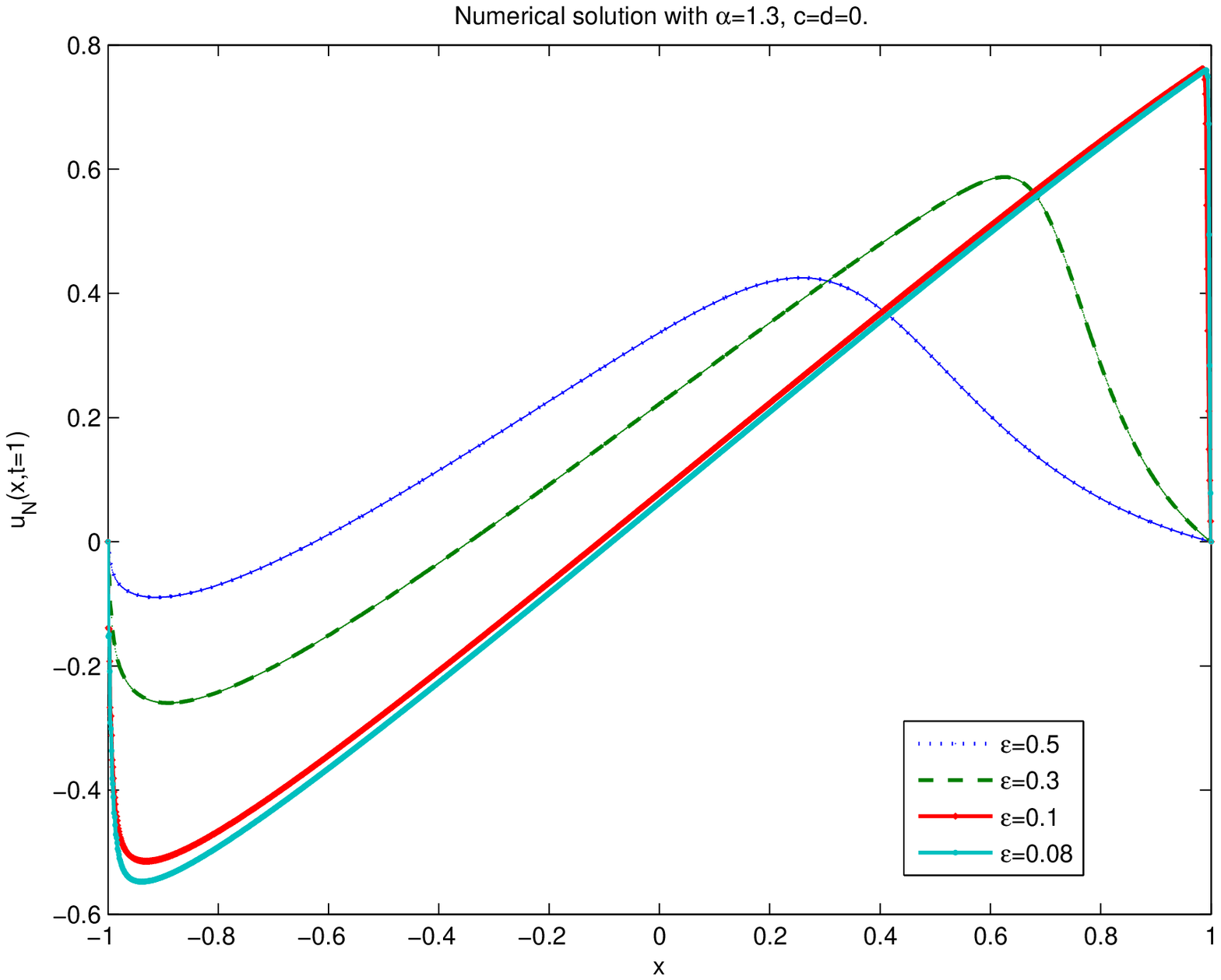}
\includegraphics[width=0.48\textwidth,height=0.4\textwidth]{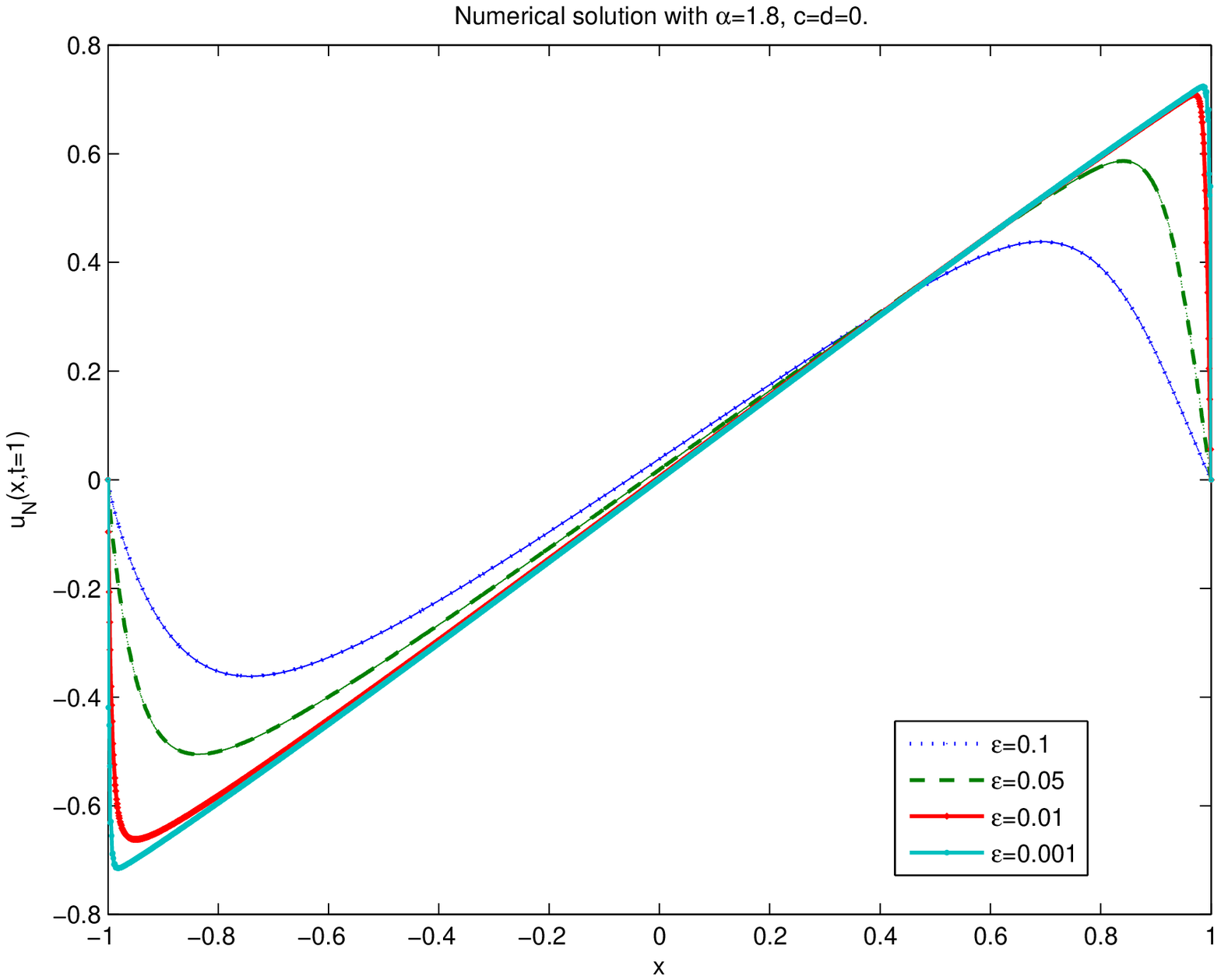}
\caption{Numerical solutions at $t=1$ for Example \ref{exp:Burgers} with the uniform mesh (Mesh 1) and different values of viscosity $\varepsilon$.  $M = 600,\, N=3, c=d=0, \tau=10^4, \Delta t=10^{-3}$. Left: $\alf=1.3$, right: $\alf=1.8$.}
\label{fig:diffviscosity}
\end{figure}


We now consider the variable-order cases, i.e., Cases 2-5. The numerical solutions at time $t = 1$ are shown in Fig. \ref{fig:variablealp}. Again, we see that the oscillations can be eliminated by refining the mesh, see, e.g, upper left and lower left plots of Fig. \ref{fig:variablealp}.

\begin{figure}[!t]
\centering
\includegraphics[width=0.48\textwidth,height=0.4\textwidth]{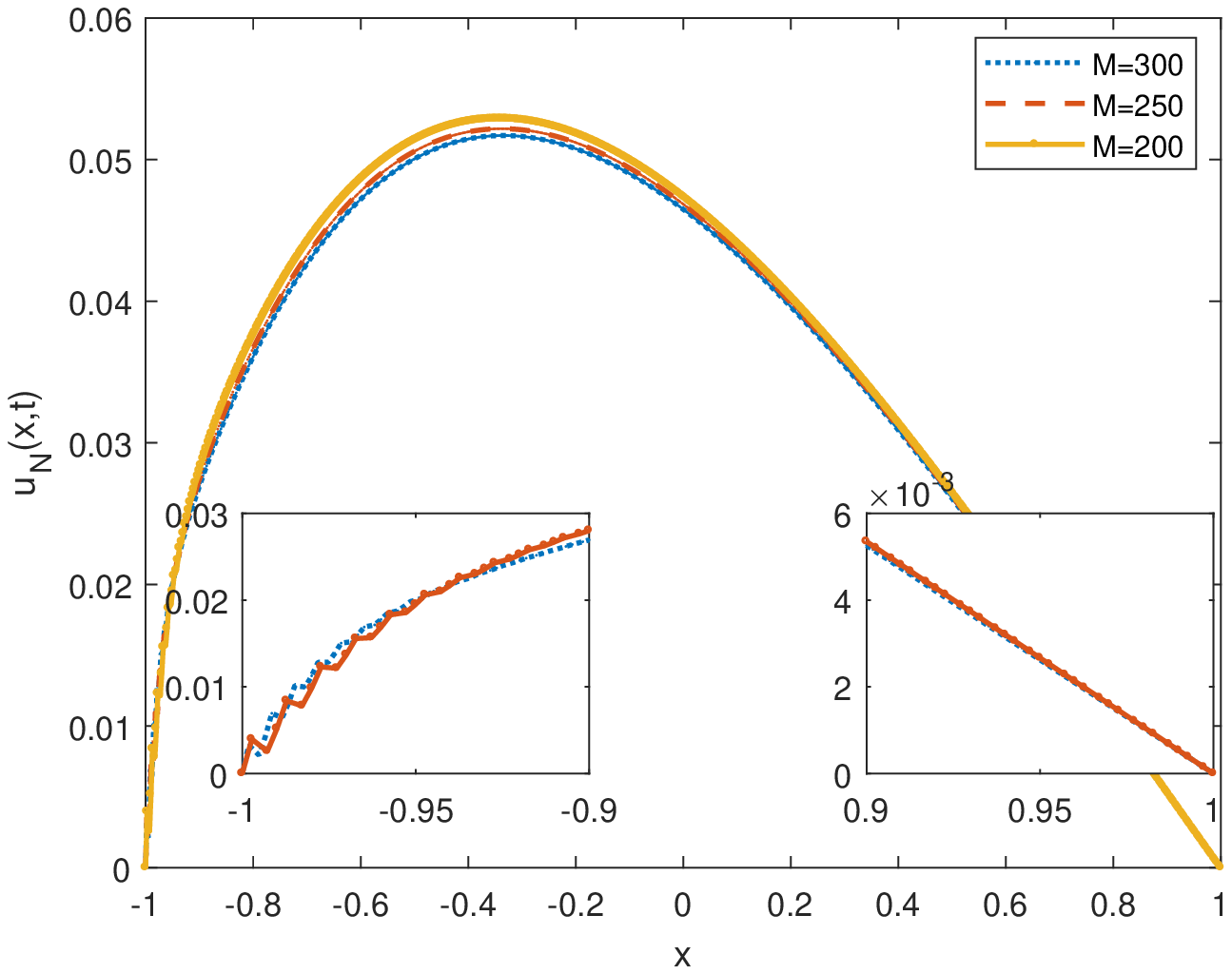}
\;
\includegraphics[width=0.48\textwidth,height=0.4\textwidth]{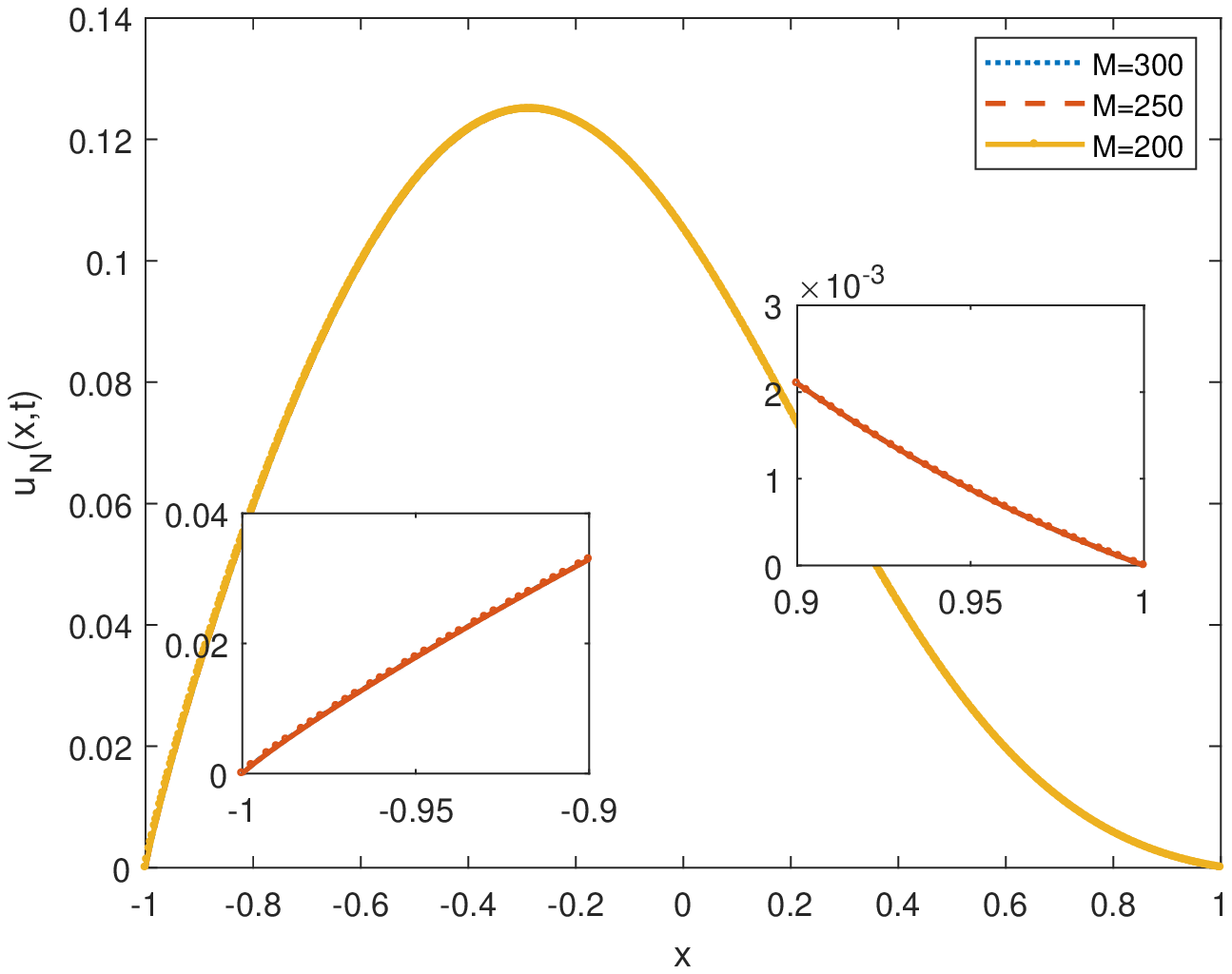}
\includegraphics[width=0.48\textwidth,height=0.4\textwidth]{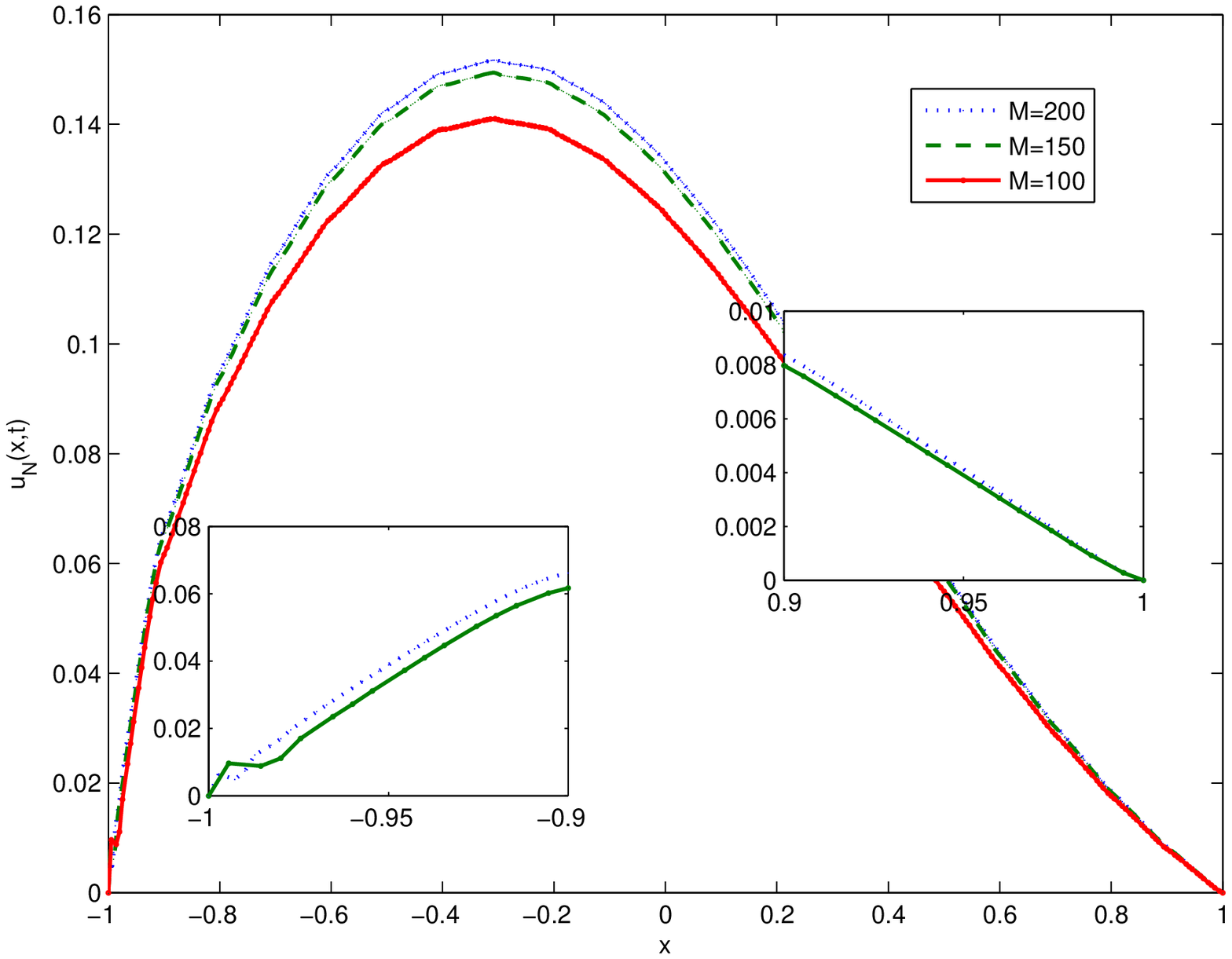}
\;
\includegraphics[width=0.48\textwidth,height=0.4\textwidth]{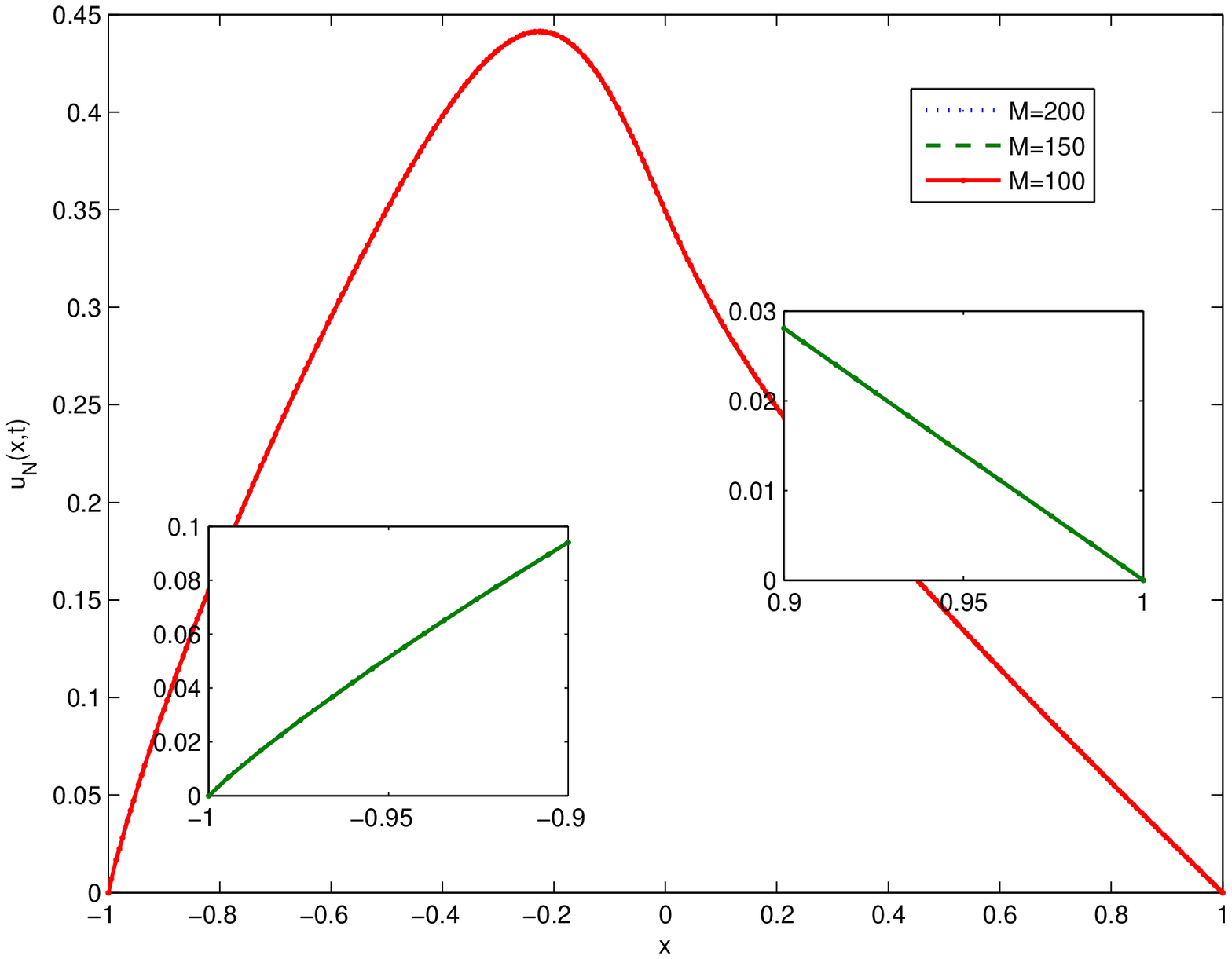}
\caption{Numerical solutions at $t=1$ for Example \ref{exp:Burgers} with the uniform mesh (Mesh 1) and $\varepsilon=1,\, N=3,\, c=d=0,\, \tau=10^5,\, \Delta t=10^{-3}$. Upper left: Case 2: $\alf=\frac{5+4x}{10}+1$, upper right: Case 3: $\alf=\frac{5-4x}{10}+1$, lower left: Case 4: $\alf=\frac{4|\sin(10\pi(x-t))|}{5}+1.1$, lower right: Case 5: $\alf=\frac{4}{5}|xt|+1.1$.}
\label{fig:variablealp}
\end{figure}


\section{Conclusion}\label{sec:conclusion}
In this paper, we present a multi-domain spectral collocation method (MDSCM) for numerically solving fractional partial differential equations that cannot be easily solved with Galerkin single- or multi-domain spectral methods. We construct a set of nodal basis functions and derive the variable-order multi-domain fractional differentiation matrix, which can be computed efficiently and are used to formulate the proposed method.
We also employ a penalty technique by minimizing the jump in (integer) fluxes to stabilize the MDSCM, which can slow down the growth of the condition number of the corresponding differentiation matrix.

Various numerical tests show that the MDSCM achieves spectral accuracy with respect to the order of polynomial under the assumption that the exact solution is sufficiently smooth.
We also demonstrate that the MDSCM has a big advantage in obtaining high accuracy when the solutions have low regularity compared with the single-domain spectral method.
Also, the multi-domain method is more flexible in performing $h$ and $p$ refinement for fractional boundary value problems as well as problems with interior regions of low regularity or very steep gradients.
In addition, it is easy to apply MDSCM to variable-coefficient problems with \textit{variable-order} fractional derivatives.
Moreover, by using the penalty method, we don't only stabilize the MDSCM but also improve greatly the accuracy of the scheme. It is especially effective for solutions with very low regularity, in which case, the non-penalized MDSCM fails to converge while the penalized MDSCM does converge in the $L^\infty$ sense. Furthermore, by choosing a suitable penalty parameter, the penalty-based MDSCM exhibits superior performance compared with the non-penalty version and the spectral element method based on the Galerkin formulation.
Unfortunately, currently we do not have a rigorous theory to show how to choose the optimal value of the penalty parameter, which is an issue of great practical interest, and this should be addressed in future work. Another open issue is the optimal penalty procedure.
In the present work, we penalize the integer fluxes, but another possibility is to penalize the fractional fluxes; however, in this case, we need to use the non-polynomial basis, namely, the poly-fractonomials proposed in \cite{ZayK13}.

\bibliographystyle{siamplain} 

\end{document}